\documentclass[a4paper,12pt]{amsart}

\usepackage{todonotes}
\usepackage{amssymb}
\usepackage{url}

\theoremstyle{thmstyletwo}%
\newtheorem{remark}{Remark}%
\newtheorem{question}{Question}

\newtheorem{defini}{Definition}
\newtheorem{thm}{Theorem}
\newtheorem{lemma}{Lemma} 
\newtheorem{corollary}{Corollary} 
\newtheorem{prop}{Proposition}

\numberwithin{equation}{section}

\title[The strange world of transfinite Melodies]{The strange world of transfinite Melodies -- Recognizability for weak and strong infinite time $\alpha$-register machines}

\author{Merlin Carl}

\begin{document}


\maketitle

\begin{abstract} 
For exponentially closed ordinals $\alpha$, we consider recognizability of constructible subsets of $\alpha$ for $\alpha$-(w)ITRMs and their distribution in the constructible hierarchy. In particular, for $\alpha$-ITRMs, we show that, there are lost melodies that are recognizable without parameters for all $\alpha$, that the iterated recognizability is absolute between $L$ and $V$ for most values of $\alpha$ and generalize ``all or nothing''-phenomenon known from ITRMs occurs for a proper class of $\alpha$. For $\alpha$-wITRMs, we offer a complete characterization of those $\alpha$ for which lost melodies exist and that the relation between the sets of computable and recognizable subsets of $\alpha$ varies wildly, depending on $\alpha$: The computable sets may be included among the recognizable sets (which is usually the case in ordinal computability), but there are also class many values of $\alpha$ for which the set of recognizable sets is empty and such for which the set of recognizable sets is non-empty, but disjoint from the set of computable sets. 

%for class many values of $\alpha$, the sets of $\alpha$-wITRM-computable and $\alpha$-wITRM-recognizable subsets of $\alpha$ are both non-empty, but disjoint, and, also for class many values of $\alpha$, the set of $\alpha$-wITRM-recognizable subsets of $\alpha$ is empty. 

This paper is an extension of our paper in the CiE 2023 proceedings \cite{cie2023}.
\end{abstract}

%KOMMENTARE:

%-Charakterisierung der $\alpha$ mit COMP$\subseteq$RECOG neu eingef\"ugt.

%-m\"ogliche Erweiterungen: Alles auch ohne Parameter (Existenz von lost melodies, f\"ur schwache Maschinen allg. Verh\"altnis RECOG UND COMP)

%-weiter: semi- und cosemierkennbarkeit einfügen.

%Der übliche Beweise zeigt für alle exponentiell abgeschlossenen $\alpha$: Auf jeder $L$-STufe sind enteweder alle $x\subseteq\alpha$ mit $x\in L_{\beta^{x}(\alpha)}$ $\alpha$-ITRM-semierkennbar oder keines davon. Für ITRM-abzählbare $\alpha$ haben alle erkennbaren diese Eigenschaft.

\section{Introduction}

%\todo{FRAGE: Gibt es erkennbare in $L$, die in $V$ nicht erkennbar sind (weil es dort weitere Zeugen gibt)? Vermutlich nicht, aber die Absolutheitss\"atze liefern das nicht mehr...}
%die Antwort ist nein, wenn man zur erkennbaren hülle übergeht: Ist x in L erkennbar,
%so wähle eine L-Stufe, die die entsprechende Berechnung enthält, und einen minimalen $\alpha$-Code c dafür. Dann ist $x\oplus c$ auch in $V$ erkennbar.

%Note: $\alpha$-wITRMs can be regarded as space-bounded ORMs. 

In \cite{HL}, Hamkins and Lewis introduced infinite time Turing machines (ITTMs), which are Turing machines that compute with transfinite time, but still on a ``standard'' tape of length $\omega$. Koepke then introduced a number of models of computation in which also memory is extended to transfinite ordinals; these include $\alpha$-ITTMs \cite{K1}, which are ITTMs with a tape of length $\alpha$, and also transfinite generalizations of register machines which can store a single ordinal less than a given ordinal $\alpha$ in each of their registers. Depending on their behaviour at limit steps, these are known as (weak) infinite time $\alpha$-register machines, or $\alpha$-(w)ITRMs for short (a brief explanation of these will be given below).

Associated with each model of computation are a concept of explicit definability -- called computability -- which concerns the ability of the machine to produce a certain object ``from scratch'', and another one of implicit definability, which concerns the ability of the machine to decide whether or not an object given in the oracle is equal to a certain $x$; in ordinal computability, the latter is known as ``recognizability''. For many models of ordinal computability, the ``lost melody phenomenon'' occurs, which means that there are objects which are recognizable, but not computable; this phenomenon was first discovered (and named) for ITTMs in \cite{HL}. Recognizability has been studied in detail for ITTMs and ITRMs in \cite{HL}, \cite{LoMe}, \cite{ITRM Recog 1}, \cite{ITRM Recog 2} and for Ordinal Turing Machines (OTMs) in \cite{CSW}. In \cite{LoMe}, we considered recognizability by $\alpha$-ITRMs. All of these works concerned real numbers, i.e., subsets of $\omega$. This has the advantage that the recognizability strength of different models becomes comparable. 

However, the natural domain of computation for $\alpha$-(w)ITRMs are clearly subsets of $\alpha$, not just subsets of $\omega$. (By analogy, the computability strength of these models (and also of $\alpha$-ITTMs) is studied in terms of subsets of $\alpha$, not of $\omega$.) The recognizability of subsets of arbitrary ordinals by OTMs with ordinal parameters is currently studied in joint work with Philipp Schlicht and Philip Welch; since already the recognizable subsets of $\omega$ can, depending on the set-theoretical background, go far beyond $L$ in this case by \cite{CSW}, it is hardly surprising that the same happens in the more general case. However, interesting phenomena also arise under the assumption $V=L$. This paper studies the recognizability strength of weak and strong $\alpha$-register machines with respect to constructible subsets of $\alpha$. While $\alpha$-ITRMs behave rather similarly to ITRMs, $\alpha$-wITRMs show a rather interesting behaviour. In particular, while the computable sets are included in the recognizable sets for all models studied so far, we will below show that, for class many values of $\alpha$, the sets of $\alpha$-wITRM-computable and $\alpha$-wITRM-recognizable subsets of $\alpha$ are both non-empty and disjoint, while for other values of $\alpha$, the set of $\alpha$-wITRM-recognizable subsets of $\alpha$ is empty. 

%We consider subsets of $\alpha$ (which is different from the approach in ``the lost melody phenomenon'', where only subsets of $\omega$ were considered) and we allow parameters (also in contrast to the other paper). 

%In [the lost melody phenomenon], this was considered for register machines without parameters and with restriction to subsets of $\omega$. This allowed for an easier comparison of the recognizability strength of various machines, as they concerned the same domain. However, it is quite natural to consider subsets of $\alpha$. ...

This paper expands our contribution to CiE 2023 \cite{cie2023} by several new main results (such as the existence of lost melodies for $\alpha$-ITRMs that are recognizable without parameters, Theorem \ref{recog closure}, Theorem \ref{distribution}(4-5), Theorem \ref{imp and recog}, Corollary \ref{comp and recog} and Corollary \ref{comp and recog without parameters}), a systematic reorganization of proofs by isolating reoccuring crucial steps as lemmata such as the existence of nice codes (Lemma \ref{if code then nice}) (which will hopefully increase the readability) and by more elaborate versions of some arguments for the sake of being self-contained.%The main results of this paper appear here as Theorem \ref{itrm lost melodies} (but it is here strengthened to the existence of lost melodies that are recognizable without parameters), Theorem \ref{distribution}(1)-(3) and most of the results on $\alpha$-wITRMs\todo{weiter ausführen!} The other results are original to this paper, unless indicated otherwise. Moreover, several proofs have been considerably reorganized by isolating reoccuring crucial steps as lemmata, which will hopefully increase the readability.

\section{Basic definitions}

In the following, $\alpha$ will denote an exponentially closed ordinal, unless explicitly stated otherwise. We briefly describe $\alpha$-ITRMs and $\alpha$-wITRMs, which were originally introduced by Koepke in \cite{K1}; full definitions can also be found in \cite{CarlBook}. 

$\alpha$-register machines use finitely many registers, which can store a single ordinal strictly less than $\alpha$ each. Programs the $\alpha$-(w)ITRMs are finite sequences of program lines, each containing one of the basic commands to increment the content of a register by $1$, to copy the content of one register to another, or to jump to a certain program line when the contents of two registers agree and continue with the next one afterwards. For reasons of technical convenience, the last command was extended a bit in \cite{Koepkes Zoo II} to allow for an instantaneous comparison of two finite sequences of registers; this modification has no influence on the results in this paper. %oracles, notation for halting programs with outputs... 
In addition, there is the oracle command: Oracles for $\alpha$-register machines are subsets $x$ of $\alpha$. Given a register index $j\in\omega$, the oracle command takes the content of the $j$-th register, say $\iota$, and then writes $1$ to the $j$th register when $\iota\in x$ and otherwise $0$. 

Infinitary register computations are now defined by recursion along the ordinals. At successor ordinals, the command in the active program line is simply carried out. (We shall assume that $\alpha$ is a limit ordinal, so that there is no question what to do when the incrementation operator increases a register content above $\alpha$.) At limit stages, the register contents and the active program line are obtained as the inferior limits of the sequences of the register contents and active program lines so far. A difficulty arises when this limit is equal to $\alpha$, and this can be solved in two ways: Either one regards the computation as undefined in such a way, thus obtaining ``weak'' or ``unresetting'' $\alpha$-register machines, called $\alpha$-wITRMs, or one resets the contents of the overflowing registers to $0$, which yields ``resetting'' or ``strong'' $\alpha$-register machines, called $\alpha$-ITRMs. When one lets $\alpha=\text{On}$, thus allowing arbitrary ordinals as register contents, one obtains Koepke's Ordinal Register Machines (ORMs), which are equal in computational power to Ordinal Turing Machines (OTMs, \cite{OTM}) and can compute exactly the constructible sets of ordinals \cite{ORM}. Thus, $\alpha$-wITRMs can be regarded as space-bounded versions of ORMs (with a constant space bound).
%Auskommentiert, da schon im Computability-Artikel über alpha-(w)ITRM-berechnungsstärke. falls der nicht genommen wird oder es da raus soll, kann es hier wieder rein. 
%\footnote{It is occasionally complained that ordinal computability, rather than providing a single, canonical model of infinitary  computability, rather provides a whole menagerie of such models. We regard this complaint as misguided. At the ``top level'', one has the confluence of models familiar from finite computability: OTMs, ORMs, ordinal $\lambda$-calculus (see Fischbach and Seyfferth \cite{lambda calculus}) all have the same computational power. The other models, such as ($\alpha$)-ITTMs, ($\alpha$-)(w)ITRMs etc. can be regarded as versions with space- or time bounds and thus as analogous to complexity classes. The difference to the finite context is that, since ordinals can have nice closure properties, constant time- and space bounds lead to convenient and interesting classes; but this is a feature, rather than a bug, of ordinal computability.} 

Thus, $\alpha$-wITRM-programs and $\alpha$-ITRM-programs are the same and both are just classical register machine programs as, e.g., introduced in \cite{Cutland}. When we write something like ``$\alpha$-wITRM-program'', we really mean that the program is intended to be run on an $\alpha$-wITRM. 

An important observation about $\alpha$-(w)ITRM-computations is the following:

\begin{defini}
In an $\alpha$-(w)ITRM-computation, a ``strong loop'' is a pair $(\iota,\xi)$ of ordinals such that the computation states -- i.e., the active program line and the register contents -- at times $\iota$ and $\xi$ are identical and such that, for every time in between, the computation states were in each component greater than or equal to these states. 
\end{defini}

It is easy to see from the liminf-rule that a strong loop will be repeated forever.

\begin{thm} [Cf. \cite{CarlBook}, generalizing \cite{KM}, Lemma $3$]
An $\alpha$-ITRM-program either halts or runs into a strong loop. 
An $\alpha$-wITRM-program either halts, runs into a strong loop or is undefined due to a register overflow. 
\end{thm}

We denote by $\chi_{x}$ the characteristic function of a set $x\subseteq\alpha$. 

\begin{defini}
A set $x\subseteq\alpha$ is $\alpha$-(w)ITRM-computable if and only if there is an $\alpha$-(w)ITRM-program $P$ and some parameter $\rho<\alpha$ such that, for each $\iota<\alpha$, $P(\iota,\rho)\downarrow=\chi_{x}(\iota)$. If $\rho=0$, we say that $x$ is $\alpha$-ITRM-computable without parameters.
\end{defini}

We will now define the concept of decidability of a set of subsets of $\alpha$. For the tape models of transfinite computations, such as ITTMs, one needs to distinguish between recognizability, semirecognizability and co-recognizability, depending on whether $\{x\}$ is decidable, semi-decidable or has a semi-decidable complement. Indeed for ITTMs, these concepts have different extensions, see \cite{CSW}. 

For ITRMs, no such distinction had to be introduced, as they exhibit a rather surprising feature: The halting problem for ITRMs with a fixed number of registers is solvable by an ITRM-program with a larger number or registers, uniformly in the oracle (see Koepke and Miller, \cite{KM}, Theorem $4$). Thus, semi-, co- and plain decidability all coincide. For wITRMs, there is even less reason for conceptual differentiation, since for these, recognizability coincides with computability, see \cite{LoMe}, \cite{CarlBook}. 

For general $\alpha$-(w)ITRMs, however, the situation is different. It is currently not known whether the bounded halting problem is solvable for $\alpha$-ITRMs unless $\alpha=\omega$ or $L_{\alpha}\models$ZF$^{-}$. Moreover, for $\alpha$-wITRMs, one needs to decide whether, in the definition of the semi-decidability of a set $x\subseteq\mathfrak{P}(\alpha)$, one allows undefined computations (i.e., computations in which an oracle overflows) or not. 

%Define semi- and corecognizability for $\alpha$-(w)ITRMs. 
%For $\alpha$-(w)ITRMs: weak and strong semi- and corecognizability: Are undefined computations allowed on the complement or not?

\begin{defini}
For $X\subseteq\mathfrak{P}(\alpha)$, let us denote by $\chi_{X}$ the characteristic function of $X$ in $\mathfrak{P}(\alpha)$. 

A set $X\subseteq\mathfrak{P}(\alpha)$ is $\alpha$-(w)ITRM-semi-decidable if and only if there are an $\alpha$-(w)ITRM-program $P$ and some $\xi<\alpha$ such that, for all $y\subseteq\alpha$, $P^{y}(\xi)\downarrow$ if and only if $y\in X$. In the case of $\alpha$-wITRMs, we demand that the computations $P^{y}(\xi)$ for $y\notin X$ do not halt, but are still defined. 

$X$ is called $\alpha$-(w)ITRM-co-semi-decidable if and only if $\mathfrak{P}(\alpha)\setminus X$ is $\alpha$-(w)ITRM-semi-decidable. 

If $X$ is both $\alpha$-(w)ITRM-semi-decidable and $\alpha$-(w)ITRM-co-semi-decidable, i.e., if there is an $\alpha$-(w)ITRM-program $P$ and some $\xi<\alpha$ such that $P^{y}(\xi)\downarrow=\chi_{X}(y)$ for all $y\subseteq\alpha$, we call $X$ $\alpha$-(w)ITRM-decidable. 

If there are an $\alpha$-wITRM-program $P$ and some $\xi<\alpha$ such that $P^{y}(\xi)\downarrow$ for all $y\in X$ and, for all $y\notin X$, $P^{y}(\xi)$ is either undefined or diverges, we call $X$ ``weakly $\alpha$-wITRM-semi-decidable''. The concept of weak $\alpha$-wITRM-co-semidecidability and of $\alpha$-wITRM-decidability are now defined in the obvious way.

If $\xi=0$, we say that $X$ is $\alpha$-(w)ITRM-(co-)(semi-)decidable etc. without parameters.
\end{defini}

\begin{defini}
Let $x\subseteq\alpha$. 

$x$ is called $\alpha$-(w)ITRM-recognizable if and only if $\{x\}$ is $\alpha$-(w)ITRM-decidable. 

$x$ is called $\alpha$-(w)ITRM-semirecognizable if and only if $\{x\}$ is $\alpha$-(w)ITRM-semidecidable. 

$x$ is called $\alpha$-(w)ITRM-cosemirecognizable if and only if $\{x\}$ is $\alpha$-(w)ITRM-co-semi-decidable. 

The weak versions of $\alpha$-wITRM-semirecognizability and $\alpha$-wITRM-co-semi-recognizability are defined in the obvious way. If the (co-)(semi-)decision algorithm uses only $0$ as a parameter, we say that $x$ is $\alpha$-(w)ITRM-(co)-(semi)-recognizable without parameters.
\end{defini}

We will use $p$ to denote Cantor's ordinal pairing function. Moreover, for a set $X$, an $\in$-formula $\phi$ and a finite tuple $\vec{p}$, we denote by $\text{Def}(X,\phi,\vec{p})$ the set $\{x\in X:(X,\in)\models\phi(x,\vec{p})\}$.%\todo{diese notation konsequent benutzen} 

When $X$ is a set and $E$ is a binary relation on $X$, then the structure $(X,E)$ can be encoded as a subset of an exponentially closed ordinal $\alpha$ (cf., e.g., \cite{CarlBook}, Def. 2.3.18) by fixing a bijection $f:\alpha\rightarrow X$ and letting $c_{f}(X,E):=\{p(\iota,\xi):\iota,\xi<\alpha\wedge f(\iota)Ef(\xi)\}$.
 In general, when $\alpha\subseteq X$, it is computationally nontrivial to identify which $\iota<\alpha$ codes a certain ordinal $\xi<\alpha$. In order to circumvent this problem, we define a class of codes for which this is trivial, as $\iota<\alpha$ is encoded by the $\iota$-th limit ordinal.

\begin{defini}{\label{nice code}}
Let $\alpha$ be exponentially closed, $X$ be a transitive set with $\alpha\subseteq X$, and let $f:\alpha\rightarrow X$ be bijective. Then $c_{f}(X,\in)$ is called an $\alpha$-\textit{code} for $X$. If $f$ is additionally such that, for all $\iota<\alpha$, we have $f(\omega\iota)=\iota$, then $c_{f}(X,\in)$ is called a \textit{nice} $\alpha$-code for $X$. 
We say that $c_{f}(X,\in)$ is \textit{very nice} if additionally $f(1)=\alpha$.

For $\alpha,\gamma\in\text{On}$, we denote by $\text{nice}_{\alpha}(\gamma)$ the set of nice $\alpha$-codes for $L_{\gamma}$. If $\text{nice}_{\alpha}(\gamma)$ contains a constructible element, then we denote by $c_{\gamma}$ its $<_{L}$-minimal element. Moreover, we let $\text{nice}_{\alpha}:=\bigcup_{\gamma\in\text{On}}\text{nice}_{\alpha}(\gamma)$. 
\end{defini}

When proving that certain subsets of $\alpha$ are not computable, it is often convenient to recall from the folklore that no %(limit)\footnote{The assumption that $\alpha$ is a limit ordinal is only made to simplify the proof.}
$L$-level can contain a nice code for itself: 

\begin{lemma}{\label{no code for itself}}
Let $\beta<\alpha$. %, $\alpha$ a limit ordinal.
Then $L_{\alpha}$ does not contain a nice $\beta$-code for $L_{\alpha}$. 
\end{lemma}
\begin{proof}
Suppose for a contradiction that $c\in\mathfrak{P}(\beta)\cap L_{\alpha}$ codes $L_{\alpha}$ via $f:\beta\rightarrow L_{\alpha}$. 
%Since $\alpha$ is a limit ordinal, there is $\delta\in(\beta,\alpha)$ such that $c\in L_{\delta}$. 
There is some $\delta<\alpha$ such that $c$ is defined over $L_{\delta}$ by some $\in$-formula (possibly with parameters in $L_{\delta}$). 

Consider the set $D:=\{\iota<\beta:p(\omega\iota,\iota)\notin c\}$. Since $c$ is definable over $L_{\delta}$, so is $D$, so that $D\in L_{\alpha}$. Pick $\xi$ such that $f(\xi)=D$. Now $\xi\in D\leftrightarrow p(\omega\xi,\xi)\notin c\leftrightarrow f(\omega\xi)\notin f(\xi)\leftrightarrow \xi\notin D$, a contradiction. 
\end{proof}

\begin{remark} 

In ordinal computability and in definability considerations, one can often switch back and forth between ordinals and the corresponding $L$-levels, so that their difference seems negligible. This is different with nice codes, %Here, we have a noteworthy exception, 
since already $L_{\omega+1}$ contains codes for all ordinals below $\omega_{1}^{\text{CK}}$, but not even a code for $L_{\omega+1}$.
\end{remark}

\begin{question}
Is it possible for an $L$-level to contain a code for itself?
\end{question}

As the following shows, there is no trifling difference between codes and nice codes concerning the question where such codes appear in the constructible hierarchy:

\begin{lemma}{\label{if code then nice}}
   Let $\alpha<\beta<\gamma$ be ordinals, where $\alpha$ is exponentially closed. Then the following are equivalent:
\begin{enumerate} 
\item $L_{\gamma+1}\setminus L_{\gamma}$ contains a bijection between $\alpha$ and $L_{\beta}$.
\item $L_{\gamma+1}\setminus L_{\gamma}$ contains an $\alpha$-code for $L_{\beta}$ .
\item $L_{\gamma+1}\setminus L_{\gamma}$ contains a nice $\alpha$-code for $L_{\beta}$.
\item $L_{\gamma+1}\setminus L_{\gamma}$ contains a very nice $\alpha$-code for $L_{\beta}$.
\end{enumerate}
\end{lemma}
\begin{proof}
   If there is a bijection $f:\alpha\rightarrow L_{\gamma}$ in $L_{\gamma+1}$, then there is some formula $\phi$ such that $f(x)=y$ if and only if $L_{\gamma}\models\phi(x,y)$ (we suppress parameters for the sake of simplicity). Then $c_{f}=\{(a,b):\exists{x,y}(\phi(a,x)\wedge\phi(b,y)\wedge x\in y)\in L_{\gamma+1}$. So (1)$\Rightarrow$(2).

   %If $c\in L_{\gamma+1}\setminus L_{\gamma}$ is a nice code for $L_{\gamma}$
   Now suppose that $c\in L_{\gamma+1}\setminus L_{\gamma}$ is an $\alpha$-code for $L_{\beta}$. In particular, this means that a new subset of $\alpha$ is generated at the $L$-level $\gamma$, which, by fine-structure, implies that a bijection between $\alpha$ and $L_{\gamma}$, and hence a bijection $f:\alpha\rightarrow L_{\beta}$, is definable over $L_{\gamma}$, and hence an element of $L_{\gamma+1}$. 

    (4)$\Rightarrow$(3) and (3)$\Rightarrow$(2) are trivial. We show that (2) implies (3).
    
    %So suppose that $c\in L_{\gamma+1}\setminus L_{\gamma}$ is an $\alpha$-code for $L_{\beta}$. In particular, this means that a new subset of $\alpha$ is generated at the $L$-level $\gamma$, which, by fine-structure, implies that a bijection between $\alpha$ and $L_{\gamma}$, and hence a bijection $f:\alpha\rightarrow L_{\beta}$, is definable over $L_{\gamma}$. 
    Suppose that (2) holds, and let $c\in L_{\gamma+1}\setminus L_{\gamma}$ be an $\alpha$-code for $L_{\beta}$. By what we just showed, we have (1), so there is a bijection $f:\alpha\rightarrow L_{\beta}$ in $L_{\gamma+1}\setminus L_{\gamma}$. 
    We want to define a new bijection $\hat{f}:\alpha\rightarrow L_{\beta}$ over $L_{\gamma}$ with the property that $\hat{f}(\omega\iota)=\iota$ for all $\iota<\alpha$. Then $c_{\hat{f}}(L_{\beta})$ will be as desired: It is definable over $L_{\gamma}$ (and hence contained in $L_{\gamma+1}$) because $\hat{f}$ is, and it is not an element of $L_{\gamma}$ by Lemma \ref{no code for itself}. Let $f=\{(x,y)\in L_{\gamma}:L_{\gamma}\models\phi_{f}(x,y)\}$ (as above, parameters are suppressed). 
    
    We start by define a surjection $f^{\prime}:\alpha\rightarrow L_{\beta}$ by letting $f^{\prime}(\omega\xi)=\xi$ and $f^{\prime}(\xi+1)=f(\xi)$ for all $\xi<\alpha$. Note that $f^{\prime}$ is still definable over $L_{\gamma}$ as 
    $$\phi_{f^{\prime}}(x,y):\Leftrightarrow \exists{\xi<\alpha}(x=\omega\xi\wedge y=\xi))\vee\exists{\xi<\alpha}(x=\xi+1\wedge\phi_{f}(\xi,y)).$$
    %by using the formula for defining $f$ to eliminate the use of $f$ as a parameter in the definition.\todo{details here!} 

    Although $f^{\prime}$ is clearly surjective (as $L_{\gamma}=f[\alpha]=f^{\prime}[\{\xi+1:\xi\in\alpha\}]$), it will not be a bijection: Every $\xi<\alpha$ has exactly two different pre-images $\omega\xi$ and $f^{-1}(\xi)+1$ under $f^{\prime}$, while every $x\in L_{\gamma}\setminus\alpha$ has precisely one preimage under $f^{\prime}$. 
    In order to obtain a bijection, we make a further modification to $f^{\prime}$. To this end, define, for a given set $x$, $\eta^{0}(x):=x$ and $\eta^{k+1}(x)=\{\eta^{k}(x)\}$ for $k\in\omega$. Note that, since $\alpha$ is a limit ordinal, we have $\eta^{k}(\iota)\in L_{\iota+k+1}\subseteq L_{\alpha}$ for every $\iota<\alpha$ and every $k\in\omega$. Moreover, the map $\eta:\omega\times\alpha\rightarrow\alpha$ which sends $(k,\iota)\in\omega\times\alpha$ to $\eta^{k}(\iota)$ is definable over $L_{\alpha}$ as $$\phi_{\eta}(k,\iota,x):\Leftrightarrow\\\exists{f}(f:k\rightarrow L_{\alpha}\wedge f(0)=\{\iota\}\wedge f(k-1)=x\wedge \forall{i<(k-1)}f(i+1)=\{f(i)\}).$$ 
    
    For every $\xi<\alpha$ such that $f^{\prime}(\xi+1)=\iota\in\alpha$, we let $\hat{f}(\xi+1):=\eta^{2}(\iota)$,\footnote{That we take $\eta^{2}(\iota)$ (i.e., $\{\{\iota\}\}$) rather than $\eta^{1}(\iota)$ (i.e., $\{\iota\}$) has the sole reason that this does not require a special treatment of $0$, as $\{\{0\}\}$ is not an ordinal, while $\{0\}=1$.} and in general $\hat{f}(f^{-1}(\eta^{k}(\iota)))=\eta^{k+1}(\iota)$ for $k\in\omega$, while for all $\xi<\alpha$ not covered by this modification, we let $\hat{f}(\xi)=f^{\prime}(\xi)$. 
    %\todo{Problem mit dieser Vorgehensweise: Wenn $\beta$ keine Limesordinalzahl ist, liegen diese geschachtelten Mengen nicht unbedingt in $L_{\beta}$. Entweder auf Limesordinalzahlen umstellen oder entsprechend rekonzipieren... -> Kein Problem, es geht ja um $\alpha$, und das ist auf jeden Fall eine Limesordinalzahl!}

    We show that $\hat{f}:\alpha\rightarrow L_{\beta}$ is indeed bijective and that it is definable over $L_{\gamma}$.

    To see that $\hat{f}$ is bijective, let $x\in L_{\gamma}$. If $x\notin \alpha$ and $x$ is not of the form $\eta^{k}(\iota)$ for some $k\in\omega\setminus\{1\}$, $\iota<\alpha$, then $|\hat{f}^{-1}(x)|=|(f^{\prime})^{-1}(x)|=|f^{-1}(x)|=1$. If $x\in\alpha$, then $|\hat{f}^{-1}(x)|=|\{\omega x\}|=1$. If $x=\eta^{k}(\iota)$ for some $1<k\in\omega$, $\iota<\alpha$, then $|\hat{f}^{-1}(x)|=|\{(f^{\prime})^{-1}(\eta^{k-1}(x)\}|=1$. %, while if $x=\{\iota\}$, then $|\hat{f}^{-1}(x)|$ is the unique element of $|\{(f^{\prime})^{-1}(\iota)\}|$ that is a successor ordinal.
    %\todo{check this carefully!} 
    So every $x\in L_{\gamma}$ has exactly one pre-image under $\hat{f}$, as desired.

    We now show that $\hat{f}$ is definable over $L_{\gamma}$. Indeed, we have $\hat{f}(x)=y$ if and only if 

    \begin{enumerate}
        \item $x$ is a limit ordinal and $y=f^{\prime}(x)$ or
        \item $\exists{\xi<\alpha}(x=\xi+1\wedge f^{\prime}(\xi+1)\in\alpha\wedge y=\eta^{2}(f^{\prime}(\xi+1)))$ or
        \item $\exists{k\in\omega}\exists{\iota<\alpha}(k\geq 2\wedge x=(f^{\prime})^{-1}(\eta^{k}(\iota))\wedge y=\eta^{k+1}(\iota)$.
    \end{enumerate}

    Since the uses of $\eta$ and $f^{\prime}$ can be eliminated using $\phi_{\eta}$ and $\phi_{f^{\prime}}$, this is expressible as an $\in$-formula over $L_{\gamma}$.

    By an obvious analogous argument (additionally letting $f^{\prime}(1)=\alpha$ and then adapting $\hat{f}$ accordingly), we obtain (2)$\Rightarrow$(4).

\end{proof}

We recall a standard definition and two observations.

\begin{defini}
For $\alpha\in\text{On}$, $\sigma_{\alpha}$ is the first stable ordinal above $\alpha$, that is, the smallest ordinal $\beta>\alpha$ such that $L_{\beta}\prec_{\Sigma_{1}}L$. 
\end{defini}

The next lemmas are part of the folklore.%\todo{Vielleicht Referenz suchen. Barwise?}

\begin{lemma}{\label{sigma alpha}}
For an ordinal $\alpha$, $\sigma_{\alpha}$ is the supremum of all ordinals $\beta$ such that, for some $\in$-formula $\phi$ and some $\xi<\alpha$, $\beta$ is minimal with the property $L_{\beta}\models\phi(\xi)$. 
\end{lemma}

\begin{lemma}{\label{sigma sequence}}
If $\alpha$ and $\beta$ are ordinals such that $\beta\in[\alpha,\sigma_{\alpha})$, then $\sigma_{\beta}=\sigma_{\alpha}$.
\end{lemma}
\begin{proof}
It is clear that $\sigma_{\alpha}\leq\sigma_{\beta}$. 

For the converse, let $\rho<\beta$ and pick $\beta^{\prime}>\rho$ such that, for some formula $\phi$ and some $\xi<\alpha$, $\beta^{\prime}$ is minimal with $L_{\beta^{\prime}}\models\phi(\xi)$. By a standard fine-structural argument, $L_{\beta^{\prime}}$ is the $\Sigma_{1}$-hull of $\xi+1$ in $L_{\beta^{\prime}}$. Thus, $\rho$ is the minimal witness for some such formula $\psi$ in some parameter $\zeta<\xi+1$. It follows that $\rho$ is the unique witness $x$ of the formula ``There is a minimal $L$-level $L_{\gamma}$ that satisfies $\phi(\xi)$ and in $L_{\gamma}$, $x$ is minimal such that $\psi(x,\zeta)$''. Consequently, every $\in$-formula in the parameter $\rho$ is equivalent to some $\in$-formula using only parameters less than $\alpha$. Thus $\sigma_{\beta}\leq\sigma_{\alpha}$.
\end{proof}

For the parameter-free case, we will also need the following:

\begin{defini}{\label{parameter-free sigma}}
    For an ordinal $\alpha$, we denote by $\sigma^{\alpha}$ the supremum of ordinals $\beta$ which are minimal with the property that, for some $\Sigma_{1}$-formula $\phi$, $L_{\beta}\models\phi(\alpha)$. 
\end{defini}

%\todo{Etwas zu der Frage sagen, ob das immer dasselbe sein muss. Vermutlich nicht. Dann Ggbsp. konstruieren.}

\begin{defini} (Cf. \cite{CarlBook}, p. 34)
An ordinal $\alpha$ is called $\alpha$-(w)ITRM-singular if and only if there are an $\alpha$-(w)ITRM-program $P$ and an ordinal $\xi<\alpha$ such that, for some $\beta<\alpha$, $P$ computes a cofinal function $f:\beta\rightarrow\alpha$ in the parameter $\xi$. 
\end{defini}

%Vermutung: Die Techniken unten zeigen in jedem Fall, dass $L_{\beta(\alpha)}$ einen $\alpha$-(w)ITRM-erkennbaren Code hat, der \"uber $L_{\beta(\alpha)}$ definierbar ist; die erste Lost Melody kommt also so fr\"uh wie nur m\"oglich. Au{\ss}erdem sollte man so einen Code aus dem Halteproblem kriegen, also ist das vermutlich auch erkennbar. 

%Lemma 3.2, Decision times: Wenn $x$ ITTM-semierkennbar ist und das Programm auf $x$ in $\alpha$ Schritten h\"alt, ist $x\in L_{\alpha^{+}}$ (dem n\"achsten admissible Level nach $\alpha$). Pr\"ufen, ob sich das verallgemeinern l\"a{\ss}t. Sicherlich wird es solange funktionieren, wie $\alpha$ in $L_{\alpha^{+}}$ abz\"ahlbar ist. 
%Evtl. reicht auch weniger, weil man $\alpha$ auf zwei mutually generische Arten abz\"ahlbar forcen kann?

We will make use of the following result of Boolos:

\begin{lemma}{\label{I am an L-level}} (\cite{Boolos}, Theorem $1^{\prime}$)
There is a parameter-free $\in$-formula $\phi$ such that, for a transitive set $X$, we have $X\models\phi$ if and only if $X$ is of the form $L_{\alpha}$ for some ordinal $\alpha$. We will call this sentence ``I am an $L$-level'' from now on. 
\end{lemma}

\section{$\alpha$-ITRMs}

In this section, we will consider recognizability of subsets of $\alpha$ for $\alpha$-ITRMs. In particular, we will prove that lost melodies exist for all exponentially closed $\alpha$.

We recall the following theorem from \cite{Koepkes Zoo II}; here, $\beta(\alpha)$ is the supremum of $\alpha$-ITRM halting times. (ZF$^{-}$ denotes Zermelo-Fraenkel set theory without the power set axiom; see \cite{GHJ} for a discussion of the axiomatizations.)

\begin{thm}{\label{itrm comp strength}}
For every $\alpha$, there is an ordinal $\gamma$ such that $x\subseteq\alpha$ is $\alpha$-ITRM-computable if and only if $x\in L_{\gamma}$. For every $\alpha$, $\gamma$ is smaller than the next $\Sigma_2$-admissible ordinal after $\alpha$. If $L_{\alpha}\models\text{ZF}^{-}$, then $\gamma=\alpha+1$. If $L_{\alpha}\not\models\text{ZF}^{-}$, then $\gamma=\beta(\alpha)$. 
\end{thm}
\begin{proof}
It is shown in \cite{Koepkes Zoo II} that $\gamma=\alpha+1$ if and only if $L_{\alpha}\models$ZF$^{-}$ (Theorem $1$) and that $\gamma$ is the supremum of the $\alpha$-ITRM-clockable ordinals otherwise (Theorem $42$). The upper bound is Corollary 3.4.13 of \cite{CarlBook}. 
\end{proof}

It was shown in \cite{Koepkes Zoo II} that, for $\alpha$ ITRM-singular, there is a lost melody for $\alpha$-ITRMs, namely the halting set (encoded as a subset of $\alpha$). We start by showing that the extra condition is in fact unnecessary and that there are constructible lost melodies for all exponentially closed ordinals $\alpha$.\footnote{The condition of exponential closure is a technical convenience; it allows us, for example, to carry out halting algorithms after each other or run nested loops of algorithms without caring for possible register overflows. We conjecture that dropping this condition would not substantially change most of the results, but merely lead to more cumbersome arguments.} 

%\begin{lemma}{\label{basis level code}}
%For any exponentially closed $\alpha$, there is an $\alpha$-ITRM-program $P_{\alpha-\text{level}}$ that computes an $\alpha$-code for $L_{\alpha}$. 
%\end{lemma}
%\begin{proof}
%...
%\end{proof}

We also recall the following generalization of results by Koepke and Seyfferth \cite{KS}, which is Theorem 2.3.28(iii) of \cite{CarlBook}.

\begin{lemma}{\label{bounded truth}}
For every exponentially closed $\alpha$ and any $n\in\omega$, there is an $\alpha$-ITRM-program $P_{\alpha-\text{ntruth}}$ such that, for every formula $\phi$ that starts with $n$ quantifier alternations, followed by a quantifier-free formula and every $x\subseteq\alpha$, $P_{\alpha-\text{ntruth}}^{x\oplus s}(\phi,x)\downarrow=1$ if and only if $\phi$ holds in the structure coded by $x$, and otherwise, $P_{\alpha-\text{ntruth}}^{x\oplus s}\downarrow=0$. 
\end{lemma}

\begin{lemma}{\label{find minimal parameter}}
Let $\alpha$ be exponentially closed, $\gamma\in\text{On}$, and let $c\subseteq\alpha$ be a nice $\alpha$-code for $L_{\gamma}$ with corresponding bijection $f:\alpha\rightarrow L_{\gamma}$. Moreover, let $\phi$, $\psi$ be $\in$-formulas. %, and let $\zeta<\alpha$. 
\begin{enumerate}
    \item The set $\{\xi<\alpha:\text{Def}(L_{\gamma},\phi,f(\xi))\models\psi\}$ %$(\zeta)\}$ 
    is $\alpha$-ITRM-decidable relative to $c$ in the parameter $\zeta$.
    \item If $X\subseteq\mathfrak{P}(\alpha)$ is $\alpha$-ITRM-decidable, then $f^{-1}(\text{min}\{\xi<\gamma:\text{Def}(L_{\gamma},\phi,\xi)\in X\})$ is uniformly $\alpha$-ITRM-computable without parameters in the oracle $c$.
\end{enumerate}
\end{lemma}
\begin{proof}
    \begin{enumerate}
        \item By Lemma \ref{bounded truth}, there is an $\alpha$-ITRM-program $P$, which, for any $\xi<\alpha$, computes $c_{\xi}:=\{\iota<\alpha:f(\iota)\in \text{Def}(L_{\gamma},\phi,f(\xi))\}$ in the oracle $c$. By niceness of $c$, it is easy to check whether all elements of $c_{\xi}$ code elements of $\alpha$ (one only needs to check whether $c_{\xi}$ consists of limit ordinals). If that is the case, then, from $c_{\xi}$, one can compute $c^{\prime}_{\xi}:=\{\iota<\alpha:\omega\iota\in c_{\xi}\}$, and then, again by Lemma \ref{bounded truth}, one can check whether the structure coded by $c_{\xi}$ is a model of $\psi$.
        \item This can be done by two nested searches through $\alpha$: For each $\zeta<\alpha$, first check, using Lemma \ref{bounded truth}, whether $\zeta$ codes an ordinal in $c$, then compute $c_{\zeta}:=\text{Def}(L_{\gamma},\phi,f(\zeta))$ (again using Lemma \ref{bounded truth} to evaluate $\phi$), then decide whether $c_{\zeta}\in X$.\footnote{What we really mean by this is that we run the decision algorithm for $X$ and, whenever a bit of the oracle is requested, we use the algorithm computing $c_{\zeta}$ to compute the relevant bit. It is, however, heuristically much easier, and also harmless, to think about this as a two-staged process in which $c_{\zeta}$ is first computed and then the decision algorithm is applied to it. We will use this manner of speaking in the following without further discussion.} If any of these is not the case, we continue with $\zeta+1$. Otherwise, we run through all $\iota<\alpha$, checking first whether $f(\iota)$ is an ordinal such that $f(\iota)<f(\zeta)$ and then whether $c_{\iota}\in X$. If such a $\iota$ is found, we continue with $\zeta+1$. Otherwise, we halt with output $\zeta$.
    \end{enumerate}
\end{proof}

\begin{lemma}{\label{wo-check}}
Let $\alpha$ be exponentially closed.%\todo{Check this. Maybe it only works for constructible subsets $x\subseteq\alpha$? Or the case that $\alpha$ is of uncountable cofinality in $L$, but not in $V$ is left open?} 
\begin{enumerate}
    \item If $\alpha$ is not a regular cardinal in $L$, then there are a constructible set $s\subseteq\alpha$ and a program $P_{\alpha-\text{WO}}$ such that, for all $x\subseteq\alpha$, $P_{\alpha-\text{WO}}^{x\oplus s}\downarrow=1$ if and only if $x$ codes a well-ordering, and otherwise, $P_{\alpha-\text{WO}}^{x\oplus s}\downarrow=0$. In fact, $s$ can be taken to be $\alpha$-ITRM-recognizable without parameters.
    \item (\cite{CarlBook}, Exercise 2.3.26) If $\alpha=\omega$ or $\text{cf}^{V}(\alpha)>\omega$, % is a regular cardinal, %in $L$
    then there is such a program $P_{\alpha-\text{WO}}$ that works in the empty oracle.
\end{enumerate}
\end{lemma}
\begin{proof}
By simply searching through $\alpha$, it is easy to check whether a given $x\subseteq\alpha$ codes a linear ordering. We are thus left with checking whether this ordering is well-founded.  
\begin{enumerate}
    \item In [\cite{CarlBook}, Theorem 2.3.25], the well-foundedness test on ITRMs by Koepke and Miller \cite{KM} is generalized to $\alpha$-ITRMs when $\alpha$ is ITRM-singular, i.e., there is an $\alpha$-ITRM-computable cofinal function $f:\beta\rightarrow\alpha$ for some $\beta<\alpha$. Since we assume that $\alpha$ is not a regular cardinal in $L$, there is a constructible cofinal function $f:\beta\rightarrow\alpha$ for some $\beta<\alpha$. Let $s=\{p(\iota,f(\iota))|\iota<\beta\}$. Then $\alpha$ is ITRM-singular in the oracle $s$, and one can easily check that the argument in \cite{CarlBook} does not depend on whether the program computing the singularization uses an oracle. 

    To see that we can take $s$ to be $\alpha$-ITRM-recognizable without parameters, let $\gamma$ be minimal %such that a function $f$ singularizing $\alpha$ is an element of $L_{\gamma}$. Hence $\gamma$ is minimal 
    with the property that $L_{\gamma}\models$``$\alpha\text{ is singular.}$'' 
    It follows by fine-structure that $L_{\gamma+1}$ contains a bijection $f:\alpha\rightarrow L_{\gamma}$ and thus, by Lemma \ref{if code then nice} a very nice code $c$ for $L_{\gamma}$. It is thus definable over $L_{\gamma}$, and we can assume without loss of generality that the definition uses a single ordinal parameter $\rho<\gamma$. Let $c=\{x\in L_{\gamma}:L_{\gamma}\models\phi(x,\rho)\}$. Without loss of generality, assume that $\rho$ is the minimal ordinal $\xi$ such that $\{x\in L_{\gamma}:L_{\gamma}\models\phi(x,\xi)\}$ is a code for $L_{\gamma}$. We claim that $c$ is $\alpha$-ITRM-recognizable without parameters.

    So let $x\subseteq\alpha$ be given in the oracle. By two nested searches through $\alpha$, we check whether, for all $\iota,\xi<\alpha$, we have $p(\iota,\omega\xi)\in x$ if and only if $\iota$ is of the form $\omega\bar{\iota}$ for some $\bar{\iota}<\xi$. We then check whether, for all $\iota<\alpha$, $p(\iota,1)$ if and only if $\iota$ is a limit ordinal. If not, $x$ cannot be a very nice $\alpha$-code, and we halt with output $0$. Otherwise, we check, using Lemma \ref{bounded truth}, whether the $\in$-structure $(X,\in)$ coded by $x$ is a model of $V=L$, of ``$\alpha$ is singular'' and of ``There is no $\delta$ such that $L_{\delta}\models$``$\alpha\text{ is singular}$''''. If not, then $x$ is not a code for $L_{\gamma}$, and we halt with output $0$. (Note that, at this point, we have not yet checked the well-foundedness of $(X,\in)$.) By searching through $\alpha$, we identify the minimal $\iota$ that codes a cofinal function $f:\beta\rightarrow\alpha$ for some $\beta<\alpha$ in $x$. More precisely, we check, for each $\sigma<\alpha$, whether $\sigma$ codes a set of ordered pairs, whether all first elements of such pairs are below some $\beta<\alpha$ (this is possible by searching through all such $\beta$), whether the set of these first elements is downwards closed and whether, for each $\delta<\alpha$, there is a second element of such a pair that is larger (which is again possible by searching through $\alpha$). This search is guaranteed to be successful, as $(X,\in)$ believes $\alpha$ to be singular. Once such a $\sigma$ has been found, we stop the search and continue. Note that at this point we know that $\alpha+1$ belongs to the well-founded part of $(X,\in)$; it follows that $\sigma$ codes an actual singularization $f$ of $\alpha$. By another search, we can identify the ordinal $\beta$ that is the domain of $f$. Now $f$ is $\alpha$-ITRM-computable from $x$: Namely, given $\iota<\beta$, we search through $\alpha$ for an element $\xi$ such that $p(p(\omega\iota,\omega\xi),\sigma)\in x$. 

    As described above, we can now use $f$ to perform well-foundedness checks on subsets of $\alpha$. In particular, once we have finished the proof of the recognizability of $c$, we know how to perform well-foundedness checks in the oracle $c$.

    In particular, we can now use $f$ to perform a well-foundedness check on $x$. If this is successful, we know that $x$ is a very nice code for $L_{\gamma}$ (otherwise, we halt with output $0$). Let $g:\alpha\rightarrow L_{\gamma}$ be the corresponding bijection. It remains to see whether $x=c$.

    Moreover, we can use $f$, along with the above procedure, to decide the set of very nice codes for $L_{\gamma}$. Let $P_{\text{vn}}$ be the $\alpha$-ITRM-program that does this. 
    
    %Also, by Lemma \ref{bounded truth}, and using that $x$ is nice, there is a program $P_{\phi}$ such that, for each $\xi<\alpha$, $P_{\phi}(\xi)$ computes $c_{\phi,\xi}:=\{\iota<\alpha:L_{\gamma}\models\phi(\iota,g(\xi))\}$. Now, for each $\xi<\alpha$, we do the following:

    %\begin{enumerate}
    %    \item Check, using Lemma \ref{bounded truth}, whether $\xi$ codes an ordinal in $x$ (i.e., whether $g(\xi)\in\text{On}$). If not, continue with $\xi+1$.
    %    \item Check, using $P_{\text{vn}}$, whether $c_{\phi,\xi}$ is a very nice code for $L_{\gamma}$. If not, continue with $\xi+1$.
    %    \item As above, check if, for some $\zeta<\alpha$, $g(\zeta)\in\text{On}$, $g(\zeta)<g(\xi)$ and $c_{\phi,\zeta}$ is a very nice code for $L_{\gamma}$. If that is the case, continue with $\xi+1$. If not, stop the search with output $\xi$.
    %\end{enumerate}

    %This will identify the minimal $\xi<\alpha$ such that $g(\xi)$ is the minimal ordinal for which $\{x\in L_{\gamma}:L_{\gamma}\models\phi(x,g(\xi))$ is a very nice $\alpha$-code for $L_{\gamma}$. 
    By Lemma \ref{find minimal parameter}, we can now identify the ordinal $\xi<\alpha$ such ghat $g(\xi)$ is the minimal ordinal for which $\text{Def}(L_{\gamma},\phi,g(\xi))$
    %$\{x\in L_{\gamma}:L_{\gamma}\models\phi(x,g(\xi))$
     is a very nice $\alpha$-code for $L_{\gamma}$. 
    
    By definition, we have $\text{Def}(L_{\gamma},\phi,g(\xi))=c$.
    %$\{x\in L_{\gamma}:L_{\gamma}\models\phi(x,g(\xi))\}=c$. 
    Using Lemma \ref{bounded truth} and the niceness of $x$, we can now compute $c$ from $x$; by comparing it bitwise to $x$, we finally determine whether $x=c$.  

    %Finally, to check whether $x=c$, we search through all $\zeta<\alpha$ once again, testing for each such $\zeta$ whether $\zeta\in x$ if and only if $L_{\gamma}\models\phi(\zeta,g(\xi))$, which can be done using Lemma \ref{bounded truth}. 
    
    %, and let $c\subseteq\alpha$ be the $<_{L}$-minimal nice $\alpha$-code of $L_{\gamma}$. We claim that $f$ is $\alpha$-ITRM-computable from $c$ and that $c$ is $\alpha$-ITRM-recognizable. 

    %\todo{DO IT!} ($f$ berechenbar: ordinalzahl $\xi$, die $f$ codiert, wird als Parameter übergeben, da $\alpha$ unter Cantor-Pairing abgeschlossen ist, kann man einfach auslesen.)

    %$c$ erkennbar: Erst prüfen, ob $f$ (also das, was in $c$ von $\xi$ codiert wird) eine Singularisierung für $\alpha$ ist. Dann das für wo-check benutzen, mit Wahrheitsprädikat prüfen, ob es eine $L$-Stufe ist und dann schauen, ob sie glaubt, dass $\alpha$ singulär ist (evtl. festlegen, dass $\alpha$ durch $1$ codiert wird; das lässt sich prüfen, weil dann alle, und nur die, $\omega\iota$ Elemente von $1$ sein dürfen) und ob es glaubt, dass es eine Stufe gibt, die das glaubt.
    \item For $\alpha=\omega$, this is proved in Koepke and Miller \cite{KM}. We now assume that $\alpha$ is uncountable.\footnote{The following argument is sketched in the hint to Exercise 2.3.26 in \cite{CarlBook} and his given here for the convenience of the reader.} %in $L$. 
    %It then suffices to assume that the %$L$-
    %cofinality of $\alpha$ is bigger than $\omega$, and we give the argument under this weaker condition.\footnote{The following argument is sketched in the hint to Exercise 2.3.26 in \cite{CarlBook} and his given here for the convenience of the reader.} 
    By assumption, every countable sequence of elements of $\alpha$ %in $L$ is bounded in $\alpha$, 
    %hence of $L$-cardinality less than $\alpha$, and thus an element of $L_{\alpha}$. Now, if $x\subseteq\alpha$ codes an ill-founded ordering, then there is a sequence $(\iota_{i}:i\in\omega)\in\alpha^{\omega}$ that encodes an infinite descending sequence in the sense of $x$. By a well-known argument, there must then be a constructible such sequence.  (Namely, if there was no such sequence in $L$, then $L$, being a model of ZFC, would contain an order isomorphism between the ordering encoded by $x$ and some ordinal, which implies that this ordering is well-founded, a contradiction.) Thus, by our initial observation, there is in fact such a sequence in $L_{\alpha}$. It thus suffices to search through $L_{\alpha}$ for such a sequence. 
    is bounded in $\alpha$. It thus suffices to search through all sequences bounded by $\beta$, for all elements of an unbounded set of ordinals $\beta$ below $\alpha$. Since $\alpha$ is exponentially closed, it suffices to consider multiplicatively closed values of $\beta$. To this end, count upwards to $\alpha$ in some separate register $R$ and consider the multiplicatively closed ordinals $\beta$ occuring in $R$. By \cite{CarlBook}, Theorem 2.3.25(ii), there is, for each such $\beta$, an $\alpha$-wITRM-program $P$ that checks subsets of $\beta$ for coding well-orderings. Moreover, this program is uniform in $\beta$. 
    
    %As an element of $L_{\alpha}$, an ill-founded sequence $s$ will actually be an element, and thus a subset, of $L_{\beta}$ for some $\beta<\alpha$. We may assume without loss of generality that $\beta$ is multiplicatively closed. By \cite{CarlBook}, Theorem 2.3.25(ii), there is, for each such $\beta$, an $\alpha$-wITRM-program $P$ that checks subsets of $\beta$ for coding well-orderings. 
    
    %We give a sketch of this search, which is based on the ideas used in Koepke [ORM-Paper]. We use a register $R_{0}$ to search through $\alpha$. Each $\iota<\alpha$ is then regarded as encoding a code for an element $x$ of $L_{\alpha}$, given in the form $(\beta,\phi,\vec{p})$, where $\beta<\alpha$, $\phi$ is (the G\"odel number of) an $\in$-formula and $\vec{p}$ is a finite sequence of ordinals below $\beta$ such that $x=\{y\in L_{\beta}:L_{\beta}\models\phi(y,\vec{p})\}$. 
    \end{enumerate}
\end{proof}

\begin{remark}
    Note that, if one assumes $V=L$, one of the cases of Lemma \ref{wo-check} will be true for any exponentially closed $\alpha\in\text{On}$. The only problematic case regular $L$-cardinals that have cofinality $\omega$ in $V$. We do not know whether a well-foundedness test exists for such ordinals.
\end{remark}

\begin{corollary}{\label{check nice names}}
Let $\alpha$ be an exponentially closed ordinal.
\begin{enumerate}
 \item If $\alpha$ that is ITRM-singular, equal to $\omega$ or satisfies $\text{cf}^{V}(\alpha)>\omega$, there is an $\alpha$-ITRM-program $P_{\text{nice}}$ that decides $\text{nice}_{\alpha}$ without parameters. %such that, for each $x\subseteq\alpha$, $P_{\text{nice}}^{x}$ halts with output $1$ or $0$, depending on whether $x$ is a nice code for an $L$-level or not. 
\item If $\alpha$ is not a regular cardinal in $L$, $\alpha=\omega$ or $\text{cf}^{V}(\alpha)>\omega$, there is a parameter-freely $\alpha$-ITRM-recognizable $s\subseteq\alpha$ such that, for some $\alpha$-ITRM-program $P_{\text{nice}}$, $P^{s}_{\text{nice}}$ decides $\text{nice}_{\alpha}$ without parameters. If $V=L$, this is true for all $\alpha$. 
\end{enumerate}    
\end{corollary}
\begin{proof}
\begin{enumerate}
    \item By Lemma \ref{wo-check}(2), it is possible to check whether $x$ codes a well-founded structure $(S,\in)$. By Lemma \ref{bounded truth} and Lemma \ref{I am an L-level}, it is possible to check whether $(S,\in)$ is an $L$-level. To check whether $x$ is also nice, run through $\alpha$ and check, for each $\iota,\xi<\alpha$ whether $p(\iota,\omega\xi)\in x$ if and only if $\iota$ is a limit ordinal smaller than $\omega\xi$.
    \item The possibility of a well-foundedness check relative to some $\alpha$-ITRM-recognizable $s\subseteq\alpha$ is now guaranteed by Lemma \ref{wo-check}(1). The rest then works as in (1). 
    If $V=L$, then every exponentially closed ordinal is either equal to $\omega$, not a regular cardinal in $L$ or has a cofinality $>\omega$, so one of the conditions is met.
\end{enumerate}
\end{proof}

We will freely use the following result:

\begin{corollary}
    Let $\alpha$ be exponentially closed. There is an $\alpha$-ITRM-program $P$ such that, for any $x\subseteq\alpha$ that nicely codes a transitive $\in$-structure containing $\alpha$, $P^{x}$ halts with output $\iota$, where $\iota$ codes $\alpha$ in $x$. In particular, whenever there is a program for deciding the set of nice names, there is also one for deciding very nice names.
\end{corollary}
\begin{proof}
    $P$ works by running through $\alpha$, checking, for each $\xi<\alpha$, whether $p(\iota,\xi)\in x$ if and only if $\iota$ is a limit ordinal. Once such $\xi$ has been found, $P$ writes it to the output register and halts.
\end{proof}

\begin{lemma}{\label{minimal name identification}}
Let $\alpha$ be exponentially closed %let $X\subseteq\mathfrak{P}(\alpha)$ be $\alpha$-ITRM-decidable 
and let $\gamma$ be an ordinal such that $c_{\gamma}$ exists, $c_{\gamma}\in (L_{\gamma+1}\setminus L_{\gamma})$ and 
such that $\text{nice}_{\alpha}(\gamma)$ is $\alpha$-ITRM-decidable (possibly in some oracle $s$).
%such that $X\cap (L_{\gamma+1}\setminus L_{\gamma})\cap\text{nice}_{\alpha}(\gamma)\neq\emptyset$. 
Then some element of 
$(L_{\gamma+1}\setminus L_{\gamma})\cap\text{nice}_{\alpha}(\gamma)$ is $\alpha$-ITRM-recognizable (in the oracle $s$). If $\text{nice}_{\alpha}(\gamma)$ is $\alpha$-ITRM-decidable without parameters, then there is an element of $(L_{\gamma+1}\setminus L_{\gamma})\cap\text{nice}_{\alpha}(\gamma)$ that is $\alpha$-ITRM-recognizable without parameters.
\end{lemma}
\begin{proof}
    Let $P$ be an $\alpha$-ITRM-program that decides $\text{nice}_{\alpha}(\gamma)$. 
    %Pick $c\in X\cap (L_{\gamma+1}\setminus L_{\gamma})\cap\text{nice}_{\alpha}(\gamma)$. 
    By assumption, there is a formula $\psi^{\prime}$ such that $c_{\gamma}=\text{Def}(L_{\gamma},\psi^{\prime},\vec{p})$ %\{x\in L_{\gamma}:L_{\gamma}\models\psi^{\prime}(x,\vec{p}\}$
     for some finite $\vec{p}\subseteq L_{\gamma}$. It is well-known that this implies the existence of a formula $\psi$ such that $c_{\gamma}=\text{Def}(L_{\gamma},\psi,\vec{p})$
     %\{x\in L_{\gamma}:L_{\gamma}\models\psi^{\prime}(x,\vec{p}\}$
      for some finite $\vec{p}\subseteq \gamma$. To simplify our notation, we assume that $\vec{p}$ consists of a single element ordinal. Let $\rho<\alpha$ be the minimal ordinal such that $c:=\text{Def}(L_{\gamma},\psi,\rho)
      %\{x\in L_{\gamma}:L_{\gamma}\models\psi^{\prime}(x,\rho\}
      \in \text{nice}_{\alpha}(\gamma)$. We claim that $c$ is $\alpha$-ITRM-recognizable.

    To see this, let $x\subseteq\alpha$ be given in the oracle. Using $P$, we first decide whether $x\in\text{nice}_{\alpha}(\gamma)$. If not, we halt with output $0$. Otherwise, we know that $x$ nicely codes $L_{\gamma}$. We run through $\alpha$ in some register. For each $\iota<\alpha$, let $x_{\iota}$ be the element of $L_{\gamma}$ coded by $\iota$ in $x$. 
    Using Lemma \ref{bounded truth}, we can check whether $x_{\iota}$ is an ordinal. If not, we continue with $\iota+1$. If it is, we again use Lemma \ref{bounded truth} to check whether $d_{\iota}:=\text{Def}(L_{\gamma},\psi,x_{\iota})
    %\{x\in L_{\gamma}:L_{\gamma}\models\psi(x,x_{\iota})\}
    \subseteq\alpha$: More precisely, we use $P_{\alpha-\text{ntruth}}$ for some sufficiently large $n$ to decide, for each $\xi<\alpha$, whether or not $L_{\gamma}\models\psi(x_{\xi},x_{\iota})$, and if that is true, we check whether $\xi$ is a limit ordinal. If that is false for any $\xi$, we halt with output $0$. Otherwise, we know that $d_{\iota}\subseteq\alpha$ and we can again use $P$ to check whether $d_{\iota}\in\text{nice}_{\alpha}(\gamma)$. If not, we continue with $\iota+1$. Otherwise, we must have $x_{\iota}=\rho$, and all that remains to check is whether $d_{\iota}=x$. If that is the case, we halt with output $1$, otherwise, we halt with output $0$. 
    
    %Without loss of generality, we assume that $\vec{p}=:\rho\in\gamma$.\footnote{It is well-known (and provable by an easy induction) that we can assume all parameters occuring in the definition to be ordinals. Considering a single ordinal, rather than a finite tuple of, simplifies our notation, but has no impact on the argument.} Since 
\end{proof}

\begin{lemma}{\label{decidable model classes}}
Let $\alpha$ be an exponentially closed ordinal, and let $n\in\omega$. Let $\Theta$ be an $\alpha$-ITRM-decidable set of $\Sigma_{n}$-sentences using parameters $\leq\alpha$, and let $\gamma>\alpha$ be minimal such that $L_{\gamma}\models\Theta$. Let $s\subseteq\alpha$ be such that $P_{\text{nice}}$ works in the oracle $s$ (which includes the possibility that $s=\emptyset$). Then $\text{nice}_{\alpha}(\gamma)$ is $\alpha$-ITRM-decidable in the oracle $s$. If $\Theta$ only uses the parameter $\alpha$ and $P_{\text{nice}}$ works parameter-freely, then $\text{nice}_{\alpha}(\gamma)$ is $\alpha$-ITRM-decidable (in the oracle $s$) without parameters.
    %In particular, this holds if $L_{\gamma}$ is minimal with the property that $L_{\gamma}\models\phi(\rho)$ for some $\rho\leq\alpha$.\todo{für berechenbare satzmengen formulieren (dann geht $\Sigma_{2}$-KP etc.)
\end{lemma}
\begin{proof}
Let $x\subseteq\alpha$ be given. Using $P_{\text{nice}}$ (and the oracle $s$, if necessary), we can check whether $x$ is a nice code for some $L$-level $L_{\delta}$. By running through all $\iota<\alpha$ and checking for each whether, for all $\xi<\alpha$, $p(\xi,\iota)\in x$ if and only if $\xi$ is a limit ordinal, we can identify the ordinal $\iota_{\alpha}<\alpha$ which codes $\alpha$ in the sense of $x$ (if no such $\iota_{\alpha}$ exists, then $\delta\leq\alpha<\gamma$, so we halt with output $0$). Using $P_{\alpha-\text{n-truth}}$, we can now check, for all $\theta\in\Theta$, whether $L_{\delta}\models\theta$; note that the ordinals encoding parameters $<\alpha$ can easily be found due to $x$ being nice, while the ordinal coding $\alpha$ is at this point known to be $\iota_{\alpha}$. If this is not the case for some $\theta$, we halt with output $0$. Otherwise, we halt with output $1$.    
\end{proof}

\begin{thm}{\label{itrm lost melodies}}
For every exponentially closed $\alpha$, there is a lost melody for $\alpha$-ITRMs, i.e., a set $x\subseteq\alpha$ such that $x$ is $\alpha$-ITRM-recognizable, but not $\alpha$-ITRM-computable. In fact, we can take $x$ to be $\alpha$-ITRM-recognizable without parameters.
\end{thm}
\begin{proof}
By Theorem \ref{itrm comp strength}, every $\alpha$-ITRM-computable $x\subseteq\alpha$ will be an element of $L_{\beta(\alpha)}$. It thus suffices to find a recognizable subset of $\alpha$ that is not contained in $L_{\beta(\alpha)}$. 
%To this end, we define a (particularly) nice $\alpha$-code $c$ for an $L$-level $L_{\beta}$ with $\beta\geq\beta(\alpha)$ and then use Lemma \ref{no code for itself}. 

We split the proof into two cases, depending on whether or not $\alpha$ is a regular cardinal in $L$. 

\bigskip 

\textbf{Case 1}: $\alpha$ is not a regular cardinal in $L$. 

By Corollary \ref{check nice names}(2), pick a parameter-freely $\alpha$-ITRM-recognizable $s\subseteq\alpha$ (possibly empty) and an $\alpha$-ITRM-program $P_{\text{nice}}$ such that $P_{\text{nice}}^{s}$ decides $\text{nice}_{\alpha}$. 
We claim that there is a very nice $\alpha$-ITRM-recognizable $\alpha$-code $c$ for $L_{\beta(\alpha)}$. Let $\phi(\alpha)$ be the $\in$-sentence ``For each $\alpha$-ITRM-program $P$, $P$ either halts or runs into a strong loop''. Then $\beta(\alpha)$ is the minimal $\gamma$ with the property that $L_{\gamma}\models\phi(\alpha)$. By Lemma \ref{decidable model classes}, $\text{nice}_{\alpha}(\gamma)$ is decidable without parameters in the oracle $s$. 
Moreover, by fine-structure, $L_{\beta(\alpha)}$ is the $\Sigma_{2}$-Skolem hull of $\alpha+1$ in itself, so a bijection $f:\alpha\rightarrow L_{\beta(\alpha)}$ is definable over $L_{\beta(\alpha)}$ and thus, we have $\text{nice}_{\alpha}(\gamma)\cap (L_{\beta(\alpha)+1}\setminus L_{\beta(\alpha})\neq\emptyset$. By Lemma \ref{minimal name identification}, $\text{nice}_{\alpha}(\beta(\alpha))\cap (L_{\beta(\alpha)+1}\setminus L_{\beta(\alpha)})$ contains an element $c$ that is $\alpha$-ITRM-recognizable without parameters relative to $s$. 

But now, $s\oplus c$ is clearly $\alpha$-ITRM-recognizable: Given the input $x\subseteq\alpha$, first decompose $x:=x_{0}\oplus x_{1}$, then check whether $x_{0}=s$ and, if so, use $s$ to check whether $x_{1}=c$. Moreover, since $c\notin L_{\beta(\alpha)}$ by definition, we also have $s\oplus c\notin L_{\beta(\alpha)}$, so $s\oplus c$ is as desired.

\bigskip 

\textbf{Case 2}:\footnote{This case uses ideas similar to those used for proving the lost melody theorem for infinite time Blum-Shub-Smale machines, see \cite{itbm}.} $\alpha$ is a regular cardinal in $L$. 
In this case, we have $L_{\alpha}\models\text{ZF}^{-}$, and so it follows from Theorem \ref{itrm comp strength} that the $\alpha$-ITRM-computable subsets of $\alpha$ are exactly those in $L_{\alpha+1}$. By Lemma \ref{no code for itself}, it suffices to show that there is an $\alpha$-ITRM-recognizable nice $\alpha$-code for $L_{\alpha+1}$. 
%Let $c$ be the $<_{L}$-minimal very nice $\alpha$-code for $L_{\alpha+1}$. 
Since $L_{\alpha+2}$ clearly contains a bijection $g:\alpha\rightarrow L_{\alpha+1}$, some very nice code for $L_{\alpha+1}$ is contained in $L_{\alpha+2}$ by Lemma \ref{if code then nice}, and thus definable over $L_{\alpha+1}$. %It is not hard to see that such a code occurs in $L_{\alpha+2}$ and is thus definable over $L_{\alpha+1}$%\todo{Genau pr\"ufen. Notfalls hoch zur n\"achsten Limesstufe...}
Say such a code is defined over $L_{\alpha+1}$ by the formula $\phi$ in the parameter $\rho\in\alpha+1$.% Let $f:\alpha\rightarrow L_{\alpha+1}$ be the corresponding bijection. 
Without loss of generality, we assume that $\rho$ is minimal with the property that $c:=\text{Def}(L_{\gamma+1},\phi,\rho)$ 
%\{x\in L_{\gamma+1}:L_{\gamma+1}\models\phi(x,\rho)\}$ 
is a very nice code for $L_{\alpha+1}$.  
 %Pick $\zeta$ such that $f(\zeta)=\rho$ 
 Let $\zeta:=f^{-1}(\rho)$ (thus, either $\zeta=\omega\rho$ if $\rho<\alpha$ or $\zeta=1$ if $\rho=\alpha$). We claim that $c$ is $\alpha$-ITRM-recognizable without parameters. %in the parameter $\zeta$. 

To see this, let $d\subseteq\alpha$ be given in the oracle. We start by checking, as in the proof of Lemma \ref{wo-check}, whether each $\iota<\alpha$ is coded by $\omega\iota$ and whether $\alpha$ is coded by $1$ in $d$. 
%We start by checking whether, for each $\iota<\alpha$, we have $p(\xi,\iota)\in d$ if and only if $\xi$ is of the form $\omega\xi^{\prime}$ with $\xi^{\prime}<\iota$, which is easily done by searching through $\alpha$. If not, we halt with output $0$. Otherwise, we know that $\omega\iota$ codes $\iota$, for $\iota<\alpha$. 
%Next, we check whether $p(\xi,1)\in d$ if and only if $\xi$ is of the form $\omega\xi$ with $\xi<\alpha$, which is again easy. If not, we halt with output $0$; otherwise, we know that $1$ codes $\alpha$. 

Using bounded truth predicate evaluation from Lemma \ref{bounded truth}, we can now run through $\alpha$ and check, for each $\iota<\alpha$, whether the $\in$-structure coded by $d$ believes that $\iota$ codes an ordinal if and only if $\iota$ is either $1$ or a limit ordinal. If not, we halt with output $0$; otherwise, we know that the set of ordinals coded by $d$ is equal to $\alpha+1$. 

Again using bounded truth predicate evaluation, we check whether the structure coded by $d$ is a model of the sentence ``I am an $L$-level''. If not, we halt with output $0$. Otherwise, we know that $d$ codes $L_{\alpha+1}$, and it remains to check that $d$ is the ``right'' code. But this can be checked as in the proof of Lemma \ref{wo-check}(1) by searching through $\alpha$ for the minimal ordinal parameter for which $\phi$ defines a very nice $\alpha$-code for $L_{\alpha+1}$, and, once it is found, using it to compute $c$ and compare it to $d$.

%Running through $\alpha$ and using $\zeta$ and bounded truth predicate evaluation once more, we can now check, for each $\iota<\alpha$, whether $L_{\alpha}\models\phi(\iota,\rho)$ if and only if $\iota\in d$. If not, we halt with output $0$. Otherwise, we have $d=c$, and we halt with output $1$. 

\end{proof}

\begin{remark}
Note that, in case (2), we have obtained a lost melody at the lowest possible $L$-level, since all elements of $L_{\alpha+1}$ are $\alpha$-ITRM-computable, and hence -recognizable by computing them and then comparing them to the oracle bit by bit.\footnote{As we will see in Corollary \ref{comp not recog} below, this procedure does no longer work for $\alpha$-wITRMs.}
    %With a bit more work, one can show that, if $\alpha$ is a regular cardinal in $L$, then there is an $\alpha$-ITRM-recognizable nice $\alpha$-code $c$ for $L_{\alpha+1}$ in $L_{\alpha+2}$, providing a lost melody at the lowest level possible (since all subsets of $\alpha$ in $L_{\alpha+1}$ are still $\alpha$-ITRM-computable).
\end{remark}

%\todo{Der folgende Abschnitt braucht einige Lemmata, die erst später kommen. Daher nach hinten verschieben, genauer hinter Lemma $5$!}
For reasons explained above, we will be concerned with constructible sets in this paper. It should, however, be noted that the relation between recognizability in $L$ and in $V$ is not trivial: On the one hand, recognizable sets can be non-constructible (see \cite{CSW}). On the other hand, it is not clear that a set that is recognizable in $L$ is also recognizable in $V$: If $P$ recognizes $x$ in $L$, then $V$ may contain some $y\neq x$ such that $P^{y}\downarrow=1$, thus spoiling $P$'s ability to recognize $x$. Even for constructible sets $x$, we thus need to carefully distinguish between being $\alpha$-ITRM-recognizable and being $\alpha$-ITRM-recognizable \textit{in $L$}.

\begin{question}
    Is there an ordinal $\alpha$ and a set $x\subseteq\alpha$ such that $x$ is $\alpha$-(w)ITRM-recognizable in $L$, but not in $V$?
\end{question}

Passing over to what in \cite{CSW} is called the ``recognizable closure'', we can exclude this possibility, at least in many cases. Let us say that $x\subseteq\alpha$ belongs to the $\alpha$-(w)ITRM-recognizable closure if and only if there is $y\subseteq\alpha$ such that $x\oplus y$ is $\alpha$-(w)ITRM-recognizable, and denote this by cRECOG$^{\text{(w)ITRM}}_{\alpha}$.

%das ist im grunde schon da (Lemma 6, ``If $\alpha$ is not a regular cardinal in $L$ ...'', dort fehlt nur, dass $s$ erkennbar gewählt werden kann.
%\begin{lemma}{\label{recognizable helpful oracle}}
%    For every exponentially closed $\alpha$, there is an $\alpha$-ITRM-recognizable set $c\subseteq\alpha$ such that the following holds:
%    \begin{enumerate}
%        \item Well-foundedness is $\alpha$-ITRM-decidable relative to $c$, i.e., there is an $\alpha$-ITRM-program $P_{\alpha-\text{wf}}$ such that, for any $x\subseteq\alpha$, $P^{c}(x)$ halts with output $1$ if and only if $x$ codes a well-ordering, and otherwise, $P^{c}(x)$ halts with output $0$. 
%        \item The bounded truth predicate is $\alpha$-ITRM-decidable relative to $c$, i.e., there is an $\alpha$-ITRM-program $P_{\alpha\text{-truth}}$ such that, for all $x\subseteq\alpha$, all $\rho\in\alpha$ and all $\in$-formulas $\phi$, $P^{c}(x,\phi,\rho)$ halts with output $1$ if and only if $S\models\phi(p)$, where $S$ is the $\in$-structure $S$ coded by $x$ and $p$ is the element of $S$ coded by $\rho$.
%    \end{enumerate}
%\end{lemma}
%\begin{proof}
%    \todo{DO IT!}
%\end{proof}

%hinter Lemma 5 schieben (einschließlich der Bemerkung darüber)
\begin{thm}{\label{recog closure}}
If $\alpha$ is exponentially closed and either not a regular cardinal in $L$ or satisfies $\text{cf}^{\text{V}}(\alpha)>\omega$, and $M\models\text{ZFC}$ is a transitive class, we have (cRECOG$^{\text{ITRM}}_{\alpha})^{L}\subseteq($cRECOG$^{\text{ITRM}}_{\alpha})^{M}$.%\todo{Das stimmt nur, wenn man Wohlfundiertheit testen und Wahrheitspr\"adikate auswerten kann. Bedingungen entsprechend anpassen (siehe die Lemmata weiter unten) oder die anderen F\"alle auch ausargumentieren!}
\end{thm}
\begin{proof}
   Let $\alpha$ be an exponentially closed ordinal, and let $x\in(\text{cRECOG}^{\text{ITRM}}_{\alpha})^{L}$. Thus, there is $y\in L$ such that $x\oplus y$ is $\alpha$-ITRM-recognizable in $L$. We show that, for some $c\in\mathfrak{P}^{L}(\alpha)$, $(x\oplus y)\oplus c$ is $\alpha$-ITRM-recognizable in $V$. It then follows by absoluteness of computations that $(x\oplus y)\oplus c$ is $\alpha$-ITRM-recognizable in any transitive $M\models\text{ZFC}$ with $L\subseteq M\subseteq V$. For simplicity of notation, let us denote $x\oplus y$ by $z$.

   So suppose that $z$ is $\alpha$-ITRM-recognizable in $L$ by the program $P$ in the parameter $\rho<\alpha$. Let $\phi(\alpha)$ be the sentence $\exists{a}P^{a}(\rho)\downarrow=1$, and let $\gamma>\alpha$ be minimal such that $L_{\gamma}\models\phi(\alpha)$. By absoluteness of computations, this implies $z\in L_{\gamma}$. By Lemma \ref{check nice names}, pick an $\alpha$-ITRM-recognizable $s\subseteq\alpha$ such that $P_{\text{nice}}$ works in the oracle $s$. The only place where $s$ is used in the proof of Lemma \ref{wo-check} is in checking well-foundedness, and the proof of Lemma \ref{wo-check} shows that $s$ is $\alpha$-ITRM-recognizable in $V$.%\todo{Gründlich prüfen! Ggf. Beweis umschreiben. Sollte aber stimmen.} 

   By Lemma \ref{decidable model classes}, $\text{nice}_{\alpha}(\gamma)$ is $\alpha$-ITRM-decidable. By fine-structure and Lemma \ref{if code then nice}, there is a nice $\alpha$-code for $L_{\gamma}$ in $L_{\gamma+1}\setminus L_{\gamma}$. By Lemma \ref{minimal name identification}, there is $c\in \text{nice}_{\alpha}(\gamma)\cap L_{\gamma+1}$ which is $\alpha$-ITRM-recognizable (in the oracle $s$). It follows that $s\oplus c$ is $\alpha$-ITRM-recognizable, as in the proof of Theorem \ref{itrm lost melodies}. 
   Since $z\in L_{\gamma}$, there is some $\xi<\alpha$ which codes $z$ in the sense of $c$. But now, given the parameter $\xi$, $z$ is $\alpha$-ITRM-computable, and hence $\alpha$-ITRM-recognizable, relative to $s\oplus c$. Hence $z\oplus (s\oplus c)$ is $\alpha$-ITRM-recognizable, so $s\oplus c$ is as desired.

\end{proof}

\begin{defini}
Denote by $\rho(\alpha)$ the supremum of the set of ordinals $\beta$ for which $L_{\beta+1}\setminus L_{\beta}$ contains an $\alpha$-ITRM-recognizable subset of $\alpha$. Similarly, we write $\rho^{w}(\alpha)$ for the analogous concept for $\alpha$-wITRMs. 

Moreover, denote by $\theta(\alpha)$ and $\theta^{w}(\alpha)$ the suprema of ordinals with an $\alpha$-ITRM-computable and an $\alpha$-wITRM-computable $\alpha$-code, respectively. 
\end{defini}

\begin{remark}
It was shown in \cite{Koepkes Zoo II} that $\theta(\alpha)=\beta(\alpha)$ unless $L_{\alpha}\models\text{ZF}^{-}$.
\end{remark}

The only property of $\alpha+1$ used in the argument for the existence of a recognizable nice $\alpha$-code for $L_{\alpha+1}$ is the existence of an $\alpha$-ITRM-computable $\alpha$-code for it. The same argument hence yields: 

\begin{corollary}
If $\beta$ has an $\alpha$-ITRM-computable $\alpha$-code, then $L_{\beta}$ has an $\alpha$-ITRM-recognizable $\alpha$-code. In particular, we have $\rho(\alpha)\geq\theta(\alpha)$. 
\end{corollary}

It is shown in \cite{LoMe} that $\rho^{w}(\omega)=\beta^{w}(\omega)$. 
We currently do not know whether $\rho(\alpha)=\beta(\alpha)$ for any exponentially closed $\alpha$ (it is known to be false for $\alpha=\omega$).

%\begin{lemma}{\label{next level}}
%Let $\alpha$ be ITRM-singular. Then there is a program $P_{\text{next-level}}$ such that, for all $c\subseteq\alpha$, if $c$ is a nice $\alpha$-code for an $L$-level $L_{\delta}$, then $P^{c}_{\text{next-level}}$ computes a nice $\alpha$-code for $L_{\delta+1}$.
%\end{lemma}
%\begin{proof}
%    \todo{DO IT!}
%    Idee: Im bestehenden Code Platz schaffen (z.B. alle Nachfolger mit $\omega$ malnehmen und $1$ addieren). Den nutzen wir dann, indem $\omega\cdot p(k,\xi)+1$ als Code für $\{x\in L_{\gamma}:L_{\gamma}\models\phi_{k}(x,\xi)\}$ betrachtet wird.
%\end{proof}

%FÜR ERSTE VERSION AUSKOMMENTIEREN, SPÄTER WIEDER REIN: AB HIER!
In the first case, however, we can also get more precise information about their distribution:\footnote{The following result are the analogues of results obtained about ITRMs in \cite{ITRM Recog 1} and \cite{ITRM Recog 2} for ITRM-singular $\alpha$.}

%Note: At least (1) and (2) below also hold when $\alpha$ is a regular cardinal in $L$, but by a somewhat different %proof. 
%New strategy: Show that all $\alpha$-ITRMs can perform (1) well-foundedness checks in suitable oracles (either the %oracle singularizes $\alpha$, or $\alpha$ is a regular cardinal in $L$, in which case $L_{\alpha}$ contains all %countable subsequences of $\alpha$ and thus, well-foundedness checks can be performed by just searching through %$L_{\alpha}$, see my book) and (2) evaluations of $\Sigma_n$-truth-predicates in $\alpha$-coded structures, for %each $n\in\omega$. This suffices to recognize $L$-levels which are minimal such that certain statements hold with %parameters $<\alpha$.

\begin{defini}
We say that $x\subseteq\alpha$ is $\alpha$-\textit{ITRM-quick} if and only if $x\in L_{\beta^{x}(\alpha)}$.
\end{defini}

%\todo{hier weiter. überlegen, was ohne parameter geht: dann ist wohl $\sigma^{\alpha}$ (statt $\sigma_\alpha$ die erste Stufe, die alles enthält.}

\begin{thm}{\label{distribution}}
%Suppose that $\alpha$ is not a regular cardinal in $L$. 
Let $\alpha$ be an exponentially closed ordinal that is not a regular cardinal in $L$ or satisfies $\text{cf}^{V}(\alpha)>\omega$.
\begin{enumerate}
    \item (Cf. \cite{ITRM Recog 1}, Theorem $27$(i) for the ITRM-version) The constructible subsets of $\alpha$ that are $\alpha$-ITRM-recognizable with parameters are contained in $L_{\sigma_{\alpha}}$, and those that are $\alpha$-ITRM-recognizable without parameters are contained in $L_{\sigma^{\alpha}}$. 
    %\todo{What about $\sigma_{\alpha}$? Can this be wrong? Nein, kann es nicht, weil nach Lemma %\ref{sigma sequence} gilt, dass $\sigma_{\alpha}=\sigma_{\alpha+1}$.}
    \item (Cf. \cite{ITRM Recog 1}, Theorem $27$(ii) for the ITRM-version) $\sigma_{\alpha}$ (and $\sigma^{\alpha}$, respectively) are minimal with this property. Thus $\rho(\alpha)=\sigma_{\alpha}$. 
    \item (Cf. \cite{ITRM Recog 1}, Theorem $27$(iii)) For any $\delta<\sigma_{\alpha}$, there is a ``gap'' of length $\geq\delta$ in the $\alpha$-ITRM-recognizables; that is, there are ordinals $\beta,\gamma,\eta$ such that $\beta+\delta\leq\gamma<\eta$, $L_{\gamma}\setminus L_{\beta}$ contains no $\alpha$-ITRM-recognizable subsets of $\alpha$, $L_{\eta}\setminus L_{\gamma}$ does contain an $\alpha$-ITRM-recognizable subset of $\alpha$, and for cofinally in $\gamma$ many $\xi$, we have $(L_{\xi+1}\setminus L_{\xi})\cap\mathfrak{P}(\alpha)\neq\emptyset$. The same is true for $\alpha$-ITRM-recognizability without parameters.
    %\item The first nonrecognizable appears at level $L_{\beta(\alpha)+1}$. Unklar, im Allgemeinen vermutlich %falsch.
    %\item For each ordinal $\zeta$, either all or none of the elements of %$\mathfrak{P}(\alpha)\cap(L_{\zeta+1}\setminus L_{\zeta})$ are $\alpha$-recognizable. 
    \item (Cf. \cite{ITRM Recog 2}, Theorem 5.2 for an ITRM-version) For all $\gamma$, either all $\alpha$-ITRM-quick elements of $\mathfrak{P}(\alpha)\cap(L_{\gamma+1}\setminus L_{\gamma})$ are $\alpha$-ITRM-semirecognizable or none is. The same is true for $\alpha$-wITRMs and for $\alpha$-ITRM-semirecognizability without parameters.
    \item If $\alpha$ is ITRM-singular or $\text{cf}^{V}(\alpha)>\omega$, then, for each $\gamma$ such that $\mathfrak{P}(\alpha)\cap(L_{\gamma+1}\setminus L_{\gamma})\neq\emptyset$, $L_{\gamma+1}\setminus L_{\gamma}$ contains an $\alpha$-ITRM-recognizable subset of $\alpha$ if and only if $c_{\gamma}$ is $\alpha$-ITRM-recognizable. The same is true without parameters.
    %denote by $c_{\gamma}$\todo{mit notation für nice codes abgleichen!} the $<_{L}$-minimal nice code for $L_{\gamma}$. Then $L_{\gamma+1}\setminus L_{\gamma}$ contains an $\alpha$-ITRM-recognizable subset of $\alpha$ if and only if $c$ is $\alpha$-ITRM-recognizable.
%Der übliche Beweise zeigt für alle exponentiell abgeschlossenen $\alpha$: Auf jeder $L$-STufe sind enteweder alle $x\subseteq\alpha$ mit $x\in L_{\beta^{x}(\alpha)}$ $\alpha$-ITRM-semierkennbar oder keines davon. Für ITRM-abzählbare $\alpha$ haben alle erkennbaren diese Eigenschaft.
\end{enumerate}
\end{thm}
\begin{proof}
The proofs are adaptations of those given for ITRMs in \cite{ITRM Recog 1} and \cite{ITRM Recog 2}. We only prove the versions in which parameters are allowed; the proofs for the  parameter-free versions are entirely analogous.
\begin{enumerate}

\item Let $P$ be an $\alpha$-ITRM-program that recognizes a constructible set $X\subseteq\alpha$ in the parameter $\zeta<\alpha$. 
The statement ``There is $Y\subseteq\alpha$ such that $P^{Y}(\zeta)\downarrow=1$'' is $\Sigma_{1}$ in the parameters $\zeta$ and $\alpha$. Let us write this statement as $\exists{x,c}\phi(x,c)$, where $\phi(x,c)$ is the $\Delta_{0}$-statement ``$c$ is a halting computation of $P$ in the oracled $x$ with output $1$''. By assumption, this statement is true in $L$. Thus, by definition of $\sigma_{\alpha+1}$, it holds in $L_{\sigma_{\alpha+1}}$. So there are $x,c\in L_{\sigma_{\alpha+1}}$ such that $\phi(x,c)$. Since computations are absolute between transitive $\in$-structures, $c$ is actually such a computation, so we must have $x=Y$, so that $Y\in L_{\sigma_{\alpha+1}}$. By Lemma \ref{sigma sequence}, we have $\sigma_{\alpha}=\sigma_{\alpha+1}$, so $Y\in L_{\sigma_{\alpha}}$.
% (let us call such an ordinal $\xi$ an $\alpha$-index). 

\item Let $\beta<\sigma_{\alpha}$. We will show that there is an $\alpha$-ITRM-recognizable subset of $\alpha$ that is not contained in $L_{\beta}$. To this end, pick, by definition of $L_{\sigma_{\alpha}}$, an ordinal $\gamma\in(\beta,\sigma_{\alpha})$ such that, for some $\Sigma_{1}$-statement $\phi(\rho)$ with parameter $\rho\in\alpha$,  we have $L_{\gamma}\models\phi(\rho)$ and $\gamma$ is minimal with this property. By fine-structure, $L_{\gamma+1}$ then contains a nice $\alpha$-code for $L_{\gamma}$. 
By Lemma \ref{minimal name identification} and Lemma \ref{decidable model classes}, there is an $\alpha$-recognizable nice $\alpha$-code $c$ for $L_{\gamma}$ in $L_{\gamma+1}\setminus L_{\gamma}$. Since $\gamma>\beta$, we have $c\notin L_{\beta}$. 

\item Pick $\beta>\sigma_{\alpha}$, along with a limit $\lambda$ of $\alpha$-indices greater than $\beta+\delta$ which is at the same time $\Sigma_{2}$-admissible. Then $L_{\lambda}$ is a model of the statement $\psi(\alpha,\delta)$, which is ``There is an ordinal $\beta$ such that $\beta+\delta$ exists and all $\alpha$-ITRM-recognizable subsets of $\alpha$ are contained in $L_{\beta}$'', and also of ``There are cofinally many $\alpha$-indices''. Thus, the statement that there exists an $L$-level with these properties is true in $L$; moreover, it is easily seen to be $\Sigma_{1}$ in the parameters $\alpha$ and $\delta$. Consequently, it will be true in $L_{\sigma_{\alpha+1}}$, and hence in $L_{\sigma_{\alpha}}$. Hence, there is a $\Sigma_{2}$-admissible ordinal $\lambda\in L_{\sigma_{\alpha}}$ which is a limit of $\alpha$-indices and believes $\psi(\alpha,\delta)$. Pick a witness $\beta$ for this statement. Assume for a contradiction that some element $X$ of $(L_{\lambda}\setminus L_{\beta})\cap\mathfrak{P}(\alpha)$ is $\alpha$-ITRM-recognizable, say by the program $P$. It follows from \cite{Koepkes Zoo II}, Theorem $46$ by the $\Sigma_{2}$-admissibility of $\lambda$ that $P^{Y}$ will either halt or run into a strong loop by time $\lambda$ for all $Y\in L_{\lambda}\cap\mathfrak{P}(\alpha)$. If the latter option were true for any such $Y$, then $P^{Y}$ would actually be looping and hence $P$ could not recognize $X$. Thus $P^{Y}$ halts in less than $\lambda$ many steps for all $Y\in L_{\lambda}$. Since $L_{\lambda}$ does not believe $X$ to be recognizable by $P$, we either must have $P^{X}\downarrow=0$ or $P^{Y}\downarrow=1$ for some $Y\in L_{\lambda}\cap\mathfrak{P}(\alpha)$ different from $X$. But both options contradict the assumption that $P$ recognizes $X$ by absoluteness of computations. 

\item (For simplicity, we drop mentioning of parameters in this proof.) Let $x,y\in\mathfrak{P}(\alpha)\cap(L_{\gamma+1}\setminus L_{\gamma})$ be $\alpha$-ITRM-quick and assume that $x$ is $\alpha$-ITRM-semirecognizable. Pick a program $P$ that semi-recognizes $x$. By assumption, 
$y\in L_{\gamma+1}$, and because $x$ is $\alpha$-ITRM-quick, we have $\beta^{x}(\alpha)>\gamma+1$, and by Lemma \ref{itrm comp strength}, we have $y\in L_{\gamma+1}\subseteq L_{\beta^{x}(\alpha)}\subseteq L_{\beta^{x}(\alpha)}[x]$. So $y$ is $\alpha$-ITRM-computable from $x$. Let $Q_{xy}$ be an $\alpha$-ITRM-program that computes $y$ in the oracle $x$. By the same argument, let $Q_{yx}$ be an $\alpha$-ITRM-program that computes $x$ in the oracle $y$. To semi-recognize $y$, proceed as follows: Given $z\subseteq\alpha$, first run $P_{yx}^{z}(\iota)$ for all $\iota<\alpha$. If this does not halt (or if 
$P_{yx}^{z}(\iota)$ halts with an output different from $0$ or $1$ for some $\iota<\alpha$), then $z\neq x$, and so the machine behaves as expected. If it halts, let $z^{\prime}$ be the subset of $\alpha$ computed by this program. Run $P^{z^{\prime}}$. %\footnote{What actually happens here is that we run $P$ and, whenever $P$ asks for the $\iota$-th bit of the oracle, we run $Q_{xy}^{z}(\iota)$ to answer this; however, it is easier to imagine the computation of 
%$z^{\prime}$ as finished.} 
If this computation does not halt, then $z^{\prime}\neq x$, which means (by definition of $Q_{xy}$) that $z\neq y$. If it does halt, then we know that 
$z^{\prime}=x$, so we run $Q_{yx}^{z^{\prime}}(\iota)$ for all $\iota<\alpha$ and compare each bit to $z$. If they disagree in some place, we run into and endless loop. Otherwise, by definition of $Q_{yx}$, we know that $z=y$, so we halt.
%Der übliche Beweise zeigt für alle exponentiell abgeschlossenen $\alpha$: Auf jeder $L$-STufe sind enteweder alle $x\subseteq\alpha$ mit $x\in L_{\beta^{x}(\alpha)}$ $\alpha$-ITRM-semierkennbar oder keines davon. Für ITRM-abzählbare $\alpha$ haben alle erkennbaren diese Eigenschaft.

\item By a simple fine-structural argument, we have $c_{\gamma}\in L_{\gamma+1}$ for every such $\gamma$ (because a new subset of $\alpha$ is definable over  $L_{\gamma}$, namely $x$), while $c_{\gamma}\notin L_{\gamma}$ by Lemma \ref{no code for itself}. Thus, if $c_{\gamma}$ is $\alpha$-ITRM-recognizable, then $L_{\gamma+1}\setminus L_{\gamma}$ contains an $\alpha$-ITRM-recognizable element. 

On the other hand, suppose that $x\in\mathfrak{P}(\alpha)\cap(L_{\gamma+1}\setminus L_{\gamma})$ is $\alpha$-ITRM-recognizable, and let $P$ be an $\alpha$-ITRM-program that recognizes $x$. We want to show that $c_{\gamma}$ is $\alpha$-ITRM-recognizable. So let $y\subseteq\alpha$ be given in the oracle. 

By Lemma \ref{check nice names}, we can check whether $y$ is a nice name for an $L$-level. If not, we halt with output $0$. Otherwise, let $L_{\delta}$ be the $L$-level coded by $y$, with bijection $f:\alpha\rightarrow L_{\delta}$. We run through $\alpha$, and, for each $\xi<\alpha$, we use Lemma \ref{bounded truth} to check whether $f(\xi)\subseteq \alpha$; if so, we check whether $P^{f(\xi)}\downarrow=1$. If no such element is found, we halt with output $0$. If one is found, we check whether $L_{\delta}$ contains an $L$-level that contains $f(\xi)$. If that is the case, we halt with output $0$. Otherwise, we know that $y$ codes $L_{\gamma}$. We have now shown how to decide $\text{nice}_{\alpha}(\gamma)$. 
We know that $c_{\gamma}=\text{Def}(L_{\gamma},\phi,\rho)$ for some $\in$-formula $\phi$ and some $\rho<\gamma$. Moreover, by minimality of $c_{\gamma}$, we can assume without loss of generality that $\rho$ is minimal with the property that $\text{Def}(L_{\gamma},\phi,\rho)$ is a nice code for $L_{\gamma}$. By Lemma \ref{find minimal parameter}, we can compute from $y$ the ordinal $\zeta$ that codes $\rho$. But now, we can use $\zeta$ to compute $c_{\gamma}$ from $y$ and compare it to $y$, thus deciding whether $y=c_{\gamma}$. 

If $x$ is $\alpha$-ITRM-recognizable without parameters, then the argument shows that the same is true for $c_{\gamma}$. 

\end{enumerate} 
\end{proof}
%%BIS HIER: SPÄTER WIEDER REIN!

\begin{question}
    Is it true that, for all $\gamma$, either all $\alpha$-ITRM-quick elements of $\mathfrak{P}(\alpha)\cap(L_{\gamma+1}\setminus L_{\gamma})$ are $\alpha$-ITRM-recognizable or none is? (By \cite{ITRM Recog 2}, Theorem 5.2, this is true for ITRMs.)
\end{question}

%\begin{thm}{\label{comp with, rec without}}
%    There is an exponentially closed ordinal $\alpha$ such that, for some $x\subseteq\alpha$, $x$ is $\alpha$-ITRM-computable, $x$ is not $\alpha$-ITRM-computable without parameters and $x$ is $\alpha$-ITRM-recognizable without parameters.
%\end{thm}
%\begin{proof}
    
%\end{proof}

%BIS HIER ÜBERARBEITET; AB HIER EINFÜGEN!

%TO DO: ALL OR NOTHING, LOWEST NONRECOGNIZABLE? [KÖNNEN \alpha-Cohens über L_{\beta(\alpha)} das Halteproblem lösen?

The Jensen-Karp-theorem (\cite{JK}, section $5$) states that $\Sigma_1$-statements are absolute between $V_{\alpha}$ and $L_{\alpha}$ when $\alpha$ is a limit or admissible ordinals. From this, and the fact that $\beta^{x}(\omega)=\omega_{\omega}^{\text{CK},x}$ for all $x\subseteq\omega$ (Koepke, \cite{K1}) one obtains that, when $x$ is $\omega$-ITRM-recognizable, then $x\in L_{\omega_{\omega}^{\text{CK},x}}$ (ibid.).  It is then natural to ask whether we have in general that $x\in L_{\beta^{x}(\alpha)}$ when $x\subseteq\alpha$ is $\alpha$-ITRM-recognizable. For the general case, we can only partial information so far.%We note here that Theorem \ref{distribution} implies this to be false:

%redundant (s.u.), daher auskommentiert
%\begin{corollary}
%If $\alpha$ is a regular cardinal in $V$, then there are $\alpha$-ITRM-recognizable constructible sets $x\subseteq\alpha$ such that $x\notin L_{\beta^{x}(\alpha)}$. 
%\end{corollary}
%\begin{proof}
%By \cite{CarlBook}, Corollary 3.4.15, we have $\beta^{x}(\alpha)=\alpha^{\omega}$ when $\alpha$ is a regular cardinal. On the other hand, by Theorem \ref{distribution}(ii), there are cofinally in $\sigma_{\alpha}$ many $\gamma$ such that $L_{\gamma+1}\setminus L_{\gamma}$ contains an $\alpha$-ITRM-recognizable subsets of $\alpha$. 
%\end{proof}

%carefully check this:
%hierher!
%For ITRMs, we have that $x\in L_{\beta^{x}(\omega)}$ for any recognizable $x\subseteq\omega$. (A similar statement holds for ITTMs, where $x\in L_{\lambda^{x}}$ for all ITTM-recognizable $x\subseteq\omega$; see \todo{QUELLE}.) For the general case, we can only partial information so far.

%\todo{das gehört hinter das lost melody theorem für $\alpha$-ITRMs}
\begin{defini}
    An ordinal $\alpha$ is called ITRM-countable if and only if there is an $\alpha$-ITRM-computable bijection $f:\omega\rightarrow\alpha$.
\end{defini}

\begin{lemma}{\label{quickly constructible recognizables}}
    \begin{enumerate}
    \item If $\alpha$ is ITRM-countable and $x\subseteq\alpha$ is $\alpha$-ITRM-recognizable, then $x\in L_{\beta^{x}(\alpha)}$. 
    \item %Wie ist es für ZF$^{-}$-Ordinalzahlen, insbesondere $L$-reguläre Kardinalzahlen?
    If $\alpha$ is a regular cardinal, %in $L$, 
    %is such that $L_{\alpha}\models\text{ZF}^{-}$, 
    then there is an $\alpha$-ITRM-recognizable set $x\subseteq\alpha$ such that $x\notin L_{\beta^{x}(\alpha)}$.     
    \end{enumerate}
\end{lemma}
\begin{proof}
\begin{enumerate}
    \item As in Proposition 46 of \cite{Koepkes Zoo II}, $\beta^{x}(\alpha)$ is a limit of admissible ordinals.\footnote{We sketch the argument for the sake of the reader: Using the computable bijection between $\omega$ and $\alpha$, every subset of $\alpha$ can be encoded as a subset of $\omega$. Already on $\omega$-ITRMs, one can compute hyperjumps of real numbers by \cite{K1}, Theorem $5$. Thus, for any $y\in L_{\beta^{x}(\alpha)}[x]$, we can compute the hyperjump of $y$, and thus a code for $\omega_{1}^{\text{CK},y}$.} Let $P$ be a program that recognizes $x$. Then $L_{\omega^{x}(\beta)}[x]\models\exists{y}P^{y}\downarrow=1$, which is $\Sigma_{1}$. Thus, by the Jensen-Karp theorem, the same statement holds in $L_{\beta^{x}(\alpha)}$, so $x\in L_{\beta^{x}(\omega)}$. 
%\todo{DO IT}: Wir können Hypersprünge berechnen, also ist $\beta^{x}(\alpha)$ ein Limes von admissibles (siehe das APAL-Paper), dann Jensen-Karp.
    \item By \cite{Koepkes Zoo II}, Corollary $16$, we have $\beta^{x}(\alpha)=\alpha^{\omega}$ for all $x\subseteq\omega$ in this case. We will show that $L_{\alpha^{\omega}}$ has an $\alpha$-ITRM-recognizable very nice $\alpha$-code $c$. By Lemma \ref{no code for itself}, we then have $c\notin L_{\alpha^{\omega}}=L_{\beta^{c}(\alpha)}$. %Let $c$ be the $<_{L}$-minimal very nice $\alpha$-code for $L_{\alpha^{\omega}}$. By Lemma \ref{no code for itself}, $c\notin L_{\alpha^{\omega}}=L_{\beta^{c}(\alpha)}$. To recognize $c$ on an $\alpha$-ITRM, we proceed as follows: Let $x\subseteq\alpha$ be given in the oracle. 
    By Lemma \ref{check nice names}, we can check whether a given $x\subseteq\alpha$ is a very nice $\alpha$-code for an $L$-level $L_{\delta}$ with $\delta>\alpha$. One can now use Lemma \ref{bounded truth} to check whether $L_{\delta}$ believes that $\alpha^{n}$ exists for all $n\in\omega$, but $\alpha^{\omega}$ does not exist. We can thus decide whether a given $x\subseteq\alpha$ is a very nice code for $L_{\alpha^{\omega}}$.  %zum ersten: "für alle $n\in\omega$ ex. Fkt. f mit dom(f)=n, f(0)=1, f(i+1)=f(i)*\alpha"; zum zweiten: Es existiert kein \beta>0 so, dass \iota\alpha<\beta für alle \iota<\beta
    It is standard that a bijection from $\alpha$ to $\alpha^{\omega}$, and hence from $\alpha$ to $L_{\alpha^{\omega}}$, is definable over $L_{\alpha^{\omega}}$, so, by Lemma \ref{if code then nice}, such a code is contained in $L_{\alpha^{\omega}+1}\setminus L_{\alpha^{\omega}}$. It now follows from Lemma \ref{minimal name identification} that there is an $\alpha$-ITRM-recognizable very nice code for $L_{\alpha^{\omega}}$ in $L_{\alpha^{\omega}+1}\setminus L_{\alpha^{\omega}}$ (which is in fact $\alpha$-ITRM-recognizable without parameters).

\end{enumerate}
    
\end{proof}

\begin{remark}
Note that the program that recognizes $c$ in part (2) uses a fixed natural number $n$ of registers, and will thus halt for any input in time $<\alpha^{n+1}$ by \cite{Koepkes Zoo II}, Theorem $14$. This observation thus reveals a considerable divergence between two natural notions of complexity for constructible subsets $x$ of $\alpha$, namely the ``time complexity'' of an $\alpha$-ITRM deciding $\{x\}$ on the one hand and the constructible rank (i.e., the minimal index $\gamma$ of an $L$-level such that $x\in L_{\gamma+1}\setminus L_{\gamma}$) of $x$ on the other. 
\end{remark}

\begin{question}
    Are there ordinal $\alpha$ such that $L_{\alpha}\not\models\text{ZF}^{-}$ such that, for some $\alpha$-ITRM-recognizable $x\subseteq\alpha$, we have $x\notin L_{\beta^{x}(\alpha)}$?    
\end{question}

This leaves the case of regular cardinals in $L$ somewhat under-explored. We currently do not know the minimal ordinal $\tau$ such that $L_{\tau}$ contains all $\alpha$-ITRM-recognizable subsets of $\alpha$ in this case. However, there is a concise characterization in this case. Recall from Hamkins and Leahy \cite{Hamkins-Leahy} that, for an $\in$-structure $M$, a set $X\subseteq M$ is called \textit{implicitly definable} in $M$ if and only if, for some $\in$-formula $\phi$, $X$ is the unique subset of $M$ such that $(M,X)\models\phi(X)$.

\begin{thm}{\label{imp and recog}}
Let $\alpha$ be a regular cardinal. Then $X\subseteq\alpha$ is $\alpha$-ITRM-recognizable if and only if $X$ is implicitly definable in $L_{\alpha}$. 
\end{thm}
\begin{proof}
First suppose that $X$ is $\alpha$-ITRM-recognizable, and let $P$ be an $\alpha$-ITRM-program that recognizes $x$. 
Suppose that $P$ uses $n$ registers. By an easy relativization of the argument in \cite{Koepkes Zoo II}, Theorem $14$, $P^{x}$ halts in $<\alpha^{n}$ many steps.  
Using the ``Pull-Back'' technique from \cite{Koepkes Zoo II}, Lemma $7$, we can express the statement ``$P^{Y}$ halts in $<\alpha^{n}$ many steps'' as an $\in$-formula $\phi(Y)$ using $Y$ as an extra predicate over $(L_{\alpha},Y)$. By assumption, $X$ is unique such that $(L_{\alpha},X)\models\phi(X)$. 

Conversely, suppose that $X$ is implicitly definable in $L_{\alpha}$ by the $\in$-formula $\phi(Y)$. Pick $n$ such that $\phi$ is $\Sigma_{n}$. Then, given some $Y\subseteq\alpha$ in the oracle, a slight variant of $P_{\alpha-\text{ntruth}}$ can be used to evaluate the truth of $\phi(Y)$ in $(L_{\alpha},Y)$ relative to the oracle $Y$, where statements of the form $Y(\iota)$ are evaluted by calling the oracle with $\iota$. 
\end{proof}

We also note the following. We say that a formula is $\Sigma_{1}^{1,\alpha}$ if and only if it is of the form $\exists{X\subseteq\alpha}\phi(X)$, and that it is $\Pi_{1}^{1,\alpha}$ if and only if it is of the form $\forall{X\subseteq\alpha}\phi(X)$, where $\phi$ is a first-order $\in$-formula. If a set $Y\subseteq\alpha$ is definable both by a $\Sigma_{1}^{1,\alpha}$ and by a $\Pi_{1}^{1,\alpha}$-formula, where these formulas are evaluted in a $\in$-structure $M$, then $Y$ is called $\Delta_{1}^{1,\alpha}(M)$. 

%How far does this go?
\begin{prop}{\label{recog delta11}}
If $\alpha$ is an uncountable regular cardinal in $L$, then every $\alpha$-ITRM-recognizable $x\subseteq\alpha$ is $\Delta_{1}^{1,\alpha}(L_{\alpha})$
\end{prop}
\begin{proof} 
%Let $x\subseteq\alpha$ be $\alpha$-ITRM-recognizable, and let $P$ be an $\alpha$-ITRM-program that recognizes $x$. 
%Suppose that $P$ uses $n$ registers. By an easy relativization of the argument in \cite{Koepkes Zoo II}, Lemma $14$, $P^{x}$ halts in $<\alpha^{n}$ many steps.  
%Using the ``Pull-Back'' technique from \cite{Koepkes Zoo II}, we can express the statement ``$P^{x}$ halts in $<\alpha^{n}$ many steps'' as an $\in$-formula in the parameter $x$ over $L_{\alpha}$; 

As in the proof Theorem \ref{imp and recog}, %of more precisely, 
there is an $\in$-formula $\phi_{P}$ such that, for all $y\subseteq\alpha$, $L_{\alpha}[x]\models\phi(y)$ if and only if $P^{y}$ halts in $<\alpha^{n}$ many steps. 

Thus, we can describe $x$ as ``$\iota\in X$ if and only if $\exists{Y}(P(Y)\downarrow=1\wedge\iota\in Y)$ -- which is $\Sigma^{1,\alpha}_{1}(L_{\alpha})$ -- and as $\forall{Y}(P(Y)\downarrow=1\rightarrow \iota\in Y)$ -- which is $\Pi^{1,\alpha}_{1}(L_{\alpha})$, so $x$ is $\Delta^{1,\alpha}_{1}(L_{\alpha})$
\end{proof}

\begin{remark}
The preceding proposition fails in general. For example, $L_{\omega_{1}^{\text{CK}}}$ contains all $\Delta_{1}^{1}$-subsets of $\omega$ (see, e.g., \cite{Barwise}, Corollary 3.2), but there are $\omega$-ITRM-recognizable subsets of $\omega$ in $L_{\alpha+1}\setminus L_{\alpha}$ for cofinally in $\sigma$ many ordinals $\alpha$ occur in up to $\sigma$, which is much bigger than $\omega_{1}^{\text{CK}}$.  
%well beyond $L_{\omega_{1}^{\text{CK}}}$, which contains all $\Delta^{1}_{1}(L_{\omega})$-subsets of $\omega$ by [REFERENZ: Mansfield-Weitkamp?]
\end{remark}

%If this works, it works in general. yes, but that does not in general yield lost melodies, as this stuff is just computable

%schon in der einleitung erwähnt.
%As soon as $\alpha>\omega_{1}$ is admissible, we can recognize $0^{\sharp}$. We should probably focus on constructible sets and avoid such sledgehammers (but mention them). 

We have so far little information about where in the constructible hierarchy the first non-recognizable subsets of $\alpha$ appear. This is known to happen at the first possible level $L_{\beta(\alpha)}$ when $\alpha=\omega$, in which case a locally Cohen-generic real number over $L_{\beta(\omega)}$ provides an example, see \cite{ITRM Recog 2}, Theorem 3.8. This can be further extended to ITRM-countable values of $\alpha$. In general, however, such generics will not be constructible, and even if they are (namely, if $\alpha$ is countable in $L$), the argument for their non-recognizability no longer works in general. Hence, we ask:

\begin{question}
    Is it true for every exponentially closed ordinal $\alpha$ that $L_{\beta(\alpha)+1}$ contains a subset of $\alpha$ that is not $\alpha$-ITRM-recognizable? In the case where $L_{\alpha}\models\text{ZF}^{-}$, the same question can be asked about $L_{\alpha+2}$.
\end{question}

\section{$\alpha$-wITRMs}{\label{itrms}}

We now consider recognizability for unresetting $\alpha$-register machines; again, we are only interested in the recognizability of constructible subsets of $\alpha$. We recall from \cite{Koepkes Zoo II} that, for any $\alpha$, $\beta^{w}(\alpha)$ denotes the supremum of $\alpha$-wITRM-halting times.

A convenient feature of $\alpha$-ITRMs is their ability to ``search through $\alpha$'', i.e., count upwards in some register from $0$ on until it overflows, thus making it possible to check each element of $\alpha$ for a certain property. For unresetting machines, this obvious strategy is not available: Counting upwards in some register would lead to an overflow of that register at time $\alpha$, which results in the computation being undefined. In some cases, however, such a search is still possible. This motivates the next definition. 

\begin{defini}
An ordinal $\alpha$ is wITRM-searchable if and only if there is a halting $\alpha$-wITRM-program $P$ such that the first register used by $P$ contains each element of $\alpha$ at least once before $P$ stops. If such a program exists that works in the oracle $x\subseteq\alpha$, we call $\alpha$ wITRM-searchable in $x$. If such a program exists that works in the parameter $0$, we say that $\alpha$ is wITRM-searchable without parameters.
\end{defini}

Based on the results in \cite{Koepkes Zoo II}, we can give a full characterization of the wITRM-searchable ordinals. To this end, we recall from \cite{Koepkes Zoo II}, Definition 60 that an ordinal is called \textit{$u$-weak} if and only if any halting $\alpha$-wITRM in the empty input halts in less than $\alpha$ many steps. It is shown in \cite{Koepkes Zoo II} that all $\Pi_3$-reflecting ordinals (and hence, in particular, all $\Sigma_2$-admissible ordinals) are $u$-weak. Moreover, the following was proved in \cite{Koepkes Zoo II}:

\begin{thm} 
[\cite{Koepkes Zoo II}, Theorem 60 and 63] An ordinal is $u$-weak if and only if it is admissible and not wITRM-singular. 
%\begin{enumerate}
%    \item [Register Zoo, Theorem 56] Any $u$-weak ordinal is admissible. 
%    \item [Register Zoo, Theorem 54] An admissible ordinal is $u$-weak if and only if it is wITRM-regular. 
%\end{enumerate}
\end{thm}

\begin{lemma}{\label{search char}}
\begin{enumerate}
\item An ordinal $\alpha$ is wITRM-searchable if and only if it is not $u$-weak, i.e., if and only if $\alpha$ is admissible and wITRM-singular. 
\item If $\alpha$ is wITRM-searchable without parameters, then there is an $\alpha$-wITRM-program that halts after at least $\alpha$ many steps.
\item If a singularization for $\alpha$ is $\alpha$-wITRM-computable without parameters, then $\alpha$ is wITRM-searchable without parameters.
\end{enumerate}
\end{lemma}
\begin{proof}
\begin{enumerate}
\item Suppose first that $\alpha$ is wITRM-searchable, and let $P$ be an $\alpha$-wITRM-program that halts after writing each element of $\alpha$ to the first register $R_{1}$ at least once. Consider the slightly modified program $P^{\prime}$ that runs $P$ but, whenever the content of $R_{1}$ changes, uses a separate register to count from $0$ upwards to to the content of $R_{1}$. Clearly, $P^{\prime}$ will run for at least $\alpha$ many steps before halting, so that $\alpha$ is not $u$-weak. 

On the other hand, suppose that $\alpha$ is not $u$-weak. Thus, $\alpha$ is not admissible or wITRM-singular. In the latter case, it is immediate from the definition of wITRM-singularity that there is an $\alpha$-wITRM-computable cofinal function $f:\beta\rightarrow\alpha$ for some $\beta<\alpha$; in the former case, this is shown in [\cite{Koepkes Zoo II}, Theorem 63]. For simplicity, let us assume without loss of generality that $f$ is increasing. Consider the following algorithm, which works in the parameter $\beta$: Use some register $R_{2}$ to run through $\beta$. For each $\iota<\beta$, compute $f(\iota)$ and $f(\iota+1)$ and store them in $R_{3}$ and $R_{4}$. Copy the content of $R_{3}$ to $R_{1}$ and use $R_{1}$ to count upwards until one reaches the content of $R_{4}$. After that, reset the contents of $R_{1}$, $R_{3}$ and $R_{4}$ to $0$ and increase the content of $R_{2}$ by $1$. If $R_{2}$ contains $\beta$, halt. It is easy to see that, in this way, $R_{1}$ will contain every ordinal less than $\alpha$ at least once before halting.
\item The proof is the same as that for the first direction of (1).
\item This works as in the reverse direction of (1).
\end{enumerate}
%We distinguish two cases:

%\bigskip 

%\textbf{Case $1$}: $\alpha$ is not admissible. 

%\bigskip 

%\textbf{Case $2$}: $\alpha$ is admissible and wITRM-singular. 

%Let $f:\beta\rightarrow\alpha$ be an $\alpha$-wITRM-computable cofinal function, where $\beta<\alpha$. For simplicity, let us assume without loss of generality that $f$ is increasing. Consider the following algorithm, which works in the parameter $\beta$: Use some register $R_{2}$ to run through $\beta$. For each $\iota<\beta$, compute $f(\iota)$ and $f(\iota+1)$ and store them in $R_{3}$ and $R_{4}$. Copy the content of $R_{3}$ to $R_{1}$ and use $R_{1}$ to count upwards until one reaches the content of $R_{4}$. After that, reset the contents of $R_{1}$, $R_{3}$ and $R_{4}$ to $0$ and increase the content of $R_{2}$ by $1$. If $R_{2}$ contains $\beta$, halt. It is easy to see that, in this way, $R_{1}$ will contain every ordinal less than $\alpha$ at least once before halting.

\end{proof}

It was shown in \cite{LoMe}, Corollary $9$ (see also \cite{CarlBook}, Corollary 4.2.20) to follow from Kreisel's basis theorem that there are no lost melodies for $\omega$-wITRMs.

If $\alpha$ is an uncountable regular cardinal in $L$, 
then the lost melody theorem fails for $\alpha$-wITRMs for rather drastic reasons:

\begin{lemma}{\label{reg card witrm}}
If $\alpha$ is %$u$-weak, 
a regular cardinal in $L$,
then there are no constructible $\alpha$-wITRM-recognizable subsets of $\alpha$. 
\end{lemma}
\begin{proof}
If $\alpha$ is a regular cardinal in $L$, then it is in particular $\Sigma_{2}^{x}$-admissible for any constructible set $x\subseteq\alpha$. It is shown in [\cite{CarlBook}, Lemma 3.4.10(ii)] that this implies that, for every $x\subseteq\alpha$ and any $\alpha$-wITRM-program $P$, $P^{x}$ will either halt in $<\alpha$ many steps or not at all. 

We can now use a standard compactness argument: 
Suppose for a contradiction that $\alpha$ is regular in $L$ and that $x\subseteq\alpha$ is recognized by the  $\alpha$-wITRM-program $P$ in the parameter $\rho<\alpha$. In particular, this means that $P^{x}$ halts in $\tau<\alpha$ many steps. Since the basic command set for $\alpha$-wITRMs allows a register content to increase at most by $1$ in each step, all register contents generated by $P$ during this computation will be smaller than $\rho+\tau$. Since $\alpha$ is indecomposable, we have $\rho+\tau<\alpha$. In particular, the oracle command can only be applied to the first $\rho+\tau$ many bits of $x$. Consequently, if we flip the $(\rho+\tau+1)$th bit of $x$ to obtain $\tilde{x}$, we shall have $P(\rho)^{\tilde{x}}\downarrow=1$ and $\tilde{x}\neq x$, contradicting the assumption that $P$ recognizes $x$ in the parameter $\rho$. 
\end{proof}

%\begin{remark}
%It is shown in [Koepke's register zoo] that all $\Pi_3$-reflecting ordinals (and hence, in particular, all %$\Sigma_2$-admissible ordinals) are $u$-weak.
%\end{remark}

%Recall from \cite{Koepkes Zoo II} that an ordinal $\alpha$ is called ``$u$-weak'' if all $\alpha$-wITRM-clockable ordinals are below $\alpha$. (It is shown in \cite{Koepkes Zoo II} that all $\Pi_3$-reflecting ordinals (and hence, in particular, all $\Sigma_2$-admissible ordinals) are $u$-weak.)

%The proof of Lemma \ref{reg card witrm} now yields:

\begin{corollary}{\label{comp not recog}}
If $\alpha$ is $u$-weak, then no $\alpha$-wITRM-computable subset of $\alpha$ is $\alpha$-wITRM-recognizable. 
\end{corollary}
\begin{proof}
If $x\subseteq\alpha$ is $\alpha$-wITRM-computable, then any $\alpha$-wITRM-program $P$ running in the oracle $x$ can be simulated by another $\alpha$-wITRM-program in the empty oracle that runs $P$ and uses the computability of $x$ to supply the required answers to oracle calls. If $\alpha$ is $u$-weak, then all halting times of $\alpha$-wITRMs in $\alpha$-wITRM-computable ordinals are thus still strictly below $\alpha$. Now argue as in the second part of the proof of Lemma \ref{reg card witrm}.
\end{proof}

%neu eingefügt: 21.04.2023
The situation in Corollary \ref{comp not recog} is somewhat surprising: One would expect at least the computable objects to be recognizable simply by computing them and then comparing the result to the oracle. As we have seen, this comparison is non-trivial, and sometimes impossible, on $\alpha$-wITRMs. Using the notion of searchability, we can give a precise characterization of when this happens.%\todo{dieses ergebnis ist neu gegenüber dem CiE-Paper -- check this}

\begin{corollary}{\label{comp and recog}}
Let $\alpha$ be exponentially closed. The following are equivalent:
\begin{enumerate}
\item COMP$^{\alpha}_{\text{wITRM}}\subseteq$RECOG$^{\alpha}_{\text{wITRM}}$.
\item $0\in$RECOG$^{\alpha}_{\text{wITRM}}$ (i.e., $0$ is $\alpha$-wITRM-recognizable).
\item COMP$^{\alpha}_{\text{wITRM}}\cap$RECOG$^{\alpha}_{\text{wITRM}}\neq\emptyset$ (i.e., some $\alpha$-wITRM-computable subset of $\alpha$ is $\alpha$-wITRM-recognizable).
\item $\alpha$ is not $u$-weak.
\item $\alpha$ is wITRM-searchable.
\end{enumerate}
\end{corollary}
\begin{proof}
Since $0$ is clearly $\alpha$-wITRM-computable, it is clear that (1) implies (2). It is also clear that (2) implies (3). That (3) implies (4) is the contraposition of Corollary \ref{comp not recog}(1) . 
The equivalence of (4) and (5) is Lemma \ref{search char}. 

It remains to see that (5) implies (1). Suppose that $\alpha$ is wITRM-searchable, and let $x\subseteq\alpha$ be $\alpha$-ITRM-computable, say by the program $P$. Moreover, by searchability of $\alpha$, let $S$ be an $\alpha$-wITRM-program that halts after having written all elements of $\alpha$ to its first register at least once. We will show that $x$ is $\alpha$-wITRM-recognizable: Given $y\subseteq\alpha$ in the oracle, run $S$. Whenever $S$ changes the content of the first register, say, to $\iota$, check whether $\iota\in y$ and whether $P(\iota)\downarrow=1$. If the answers do not agree, halt with output $0$. When $S$ reaches the halting state, halt with output $1$. By the definition of $S$, this will halt with output $1$ if and only if $y=x$.
\end{proof}
%bis hier neu eingefügt

\begin{corollary}{\label{comp and recog without parameters}}
Let $\alpha$ be exponentially closed. The following are equivalent:
\begin{enumerate}
    \item Every $x\subseteq\alpha$ that is $\alpha$-wITRM-computable without parameters is $\alpha$-wITRM-recognizable without parameters.
    \item $0$ is $\alpha$-wITRM-recognizable without parameters.
    \item Some $x\subseteq\alpha$ that is $\alpha$-wITRM-computable without parameters is $\alpha$-wITRM-recognizable without parameters.
    \item $\alpha$ is wITRM-searchable without parameters.
\end{enumerate}
\end{corollary}
\begin{proof}
    The implications between (1), (2), (3) and the implication (4)$\Rightarrow$(1) work as in the proof of Corollary \ref{comp and recog}. We show that (2)$\Rightarrow$(4). Let $P$ be an $\alpha$-wITRM-program that recognizes $0$. Then $P$ halts in the empty oracle. If there was some $\iota<\alpha$ such that $P^{0}$ makes no oracle call with $\iota$, then we would have $P^{\{\iota\}}\downarrow=1$, contradicting the assumption that $P$ recognizes $0$. So the sequence of oracle calls performed by $P^{0}$ constitutes a search through $\alpha$.
\end{proof}

We recall some results from \cite{CarlBook}, which in turn are generalizations of results from Koepke \cite{K} (pp. 261f).

\begin{lemma}{\label{identification lemma}}
Let $\alpha$ be exponentially closed, $\beta<\alpha$ and $c$ a nice $\beta$-code for a (transitive) $\in$-structure $S\supseteq\alpha$ via some bijection $f:\beta\rightarrow S$. 

\begin{enumerate} 
\item (Cf. \cite{CarlBook}, Lemma 2.3.29, generalizing \cite{K}, p. 261f.) There is an $\alpha$-wITRM-program $P_{\text{compare}}$ such that, when $c,c^{\prime}\subseteq\beta$ are nice $\beta$-codes for (transitive) $\in$-structures and $\iota,\iota^{\prime}<\beta$ are such that $\iota$ codes an ordinal in $c$ and $\iota^{\prime}$ codes an ordinal in $c^{\prime}$, then $P_{\text{compare}}^{c\oplus c^{\prime}}(\iota,\iota^{\prime})$ decides whether $\iota$ codes the same ordinal in $c$ that $\iota^{\prime}$ codes in $c^{\prime}$. 
\item There is an $\alpha$-wITRM-program $P_{\text{identify}}$ such that, 
%for every $\beta<\alpha$, every nice code $c\subseteq\beta$ for an $\in$-structure $S\supseteq\alpha$ via some bijection $f:\beta\rightarrow S$ and 
for every $\iota<\alpha$, 
$P_{\text{identify}}^{c}(\iota,\beta)$ halts with output $f^{-1}(\iota)$ (i.e., the ordinal that codes $\iota$ in the sense of $c$). 

\item Moreover, there is an $\alpha$-wITRM-program $P_{\text{decode}}$ such that, for every $\iota<\beta$ such that $f(\omega\iota)\in\alpha$, $P_{\text{decode}}^{c}(\omega\iota,\beta)$ halts with output $f(\omega\iota)$, i.e., with output $\iota$.%\todo{Zun\"achst $\omega$ ``herausteilen''?}  
\end{enumerate}
\end{lemma} 
\begin{proof}
In \cite{CarlBook}, Lemma 2.3.29 and Corollary 2.3.31, it is shown that this can be achieved when $\iota<\beta$. However, this condition can be met by simply making $\beta$ larger if necessary and regarding $c$ as a subset of the increased $\beta$.

%Kurzform: By \cite{CarlBook}, Corollary 2.3.31, this can be done when $\iota<\beta$. However, this condition can be met by simply making $\beta$ larger if necessary. 

%This is a slightly more general variant of \cite{CarlBook}, Corollary 2.3.31. 
We sketch the (slighly generalized) algorithms described in \cite{CarlBook} for the convenience of the reader.

Note that $P_{\text{identify}}$ is easily obtained from $P_{\text{compare}}$: We fix a code $c_{\iota+1}:=\{p(\iota_{1},\iota_{2}):\iota_{1}<\iota_{2}<\iota+1\}$ for $\iota+1$; $c_{\iota}$ is clearly computable on an $\alpha$-wITRM. Then, we can run $P^{c\oplus c_{\iota+1}}(\xi,\iota)$ for every $\xi<\alpha$ and output $\xi$ as soon as the answer is positive. (Note that this search is guaranteed to terminate by our assumptions, so that the search does not lead to an overflow.)

Likewise, given the program $P_{\text{identify}}$ of (2), $P_{\text{decode}}$ is easily obtained: Given $\iota<\beta$, first compute $\omega\iota$, then use an extra register to run through $\alpha$ and, for each $\xi<\alpha$, apply $P_{\text{identify}}$ to check whether $\omega\iota$ codes $\xi$. By assumption, such a $\xi$ will eventually be considered, in which case the algorithm stops and outputs $\xi$. 

It thus remains to sketch $P_{\text{compare}}$. So let $c,c^{\prime}\subseteq\beta$ and $\iota,\iota^{\prime}<\beta$ be given.  Note that, by assumption, $c$ and $c^{\prime}$ code well-founded structures, say via bijections $f:\beta\rightarrow S$, $f^{\prime}\rightarrow S^{\prime}$.

We work with two main registers $R$ and $R^{\prime}$, both of which are intended store finite sequences of ordinals below $\text{max}\{\beta,\iota\}$, encoded via iterated Cantor pairing; let us write $p(\iota_{0},...,\iota_{k})$ for this code. In order to ensure compatibility with inferior limits -- i.e., in order to ensure that $\text{liminf}_{\xi<\lambda}p(\iota_{0},...,\iota_{k,\xi})=p(\iota_{0},...,\text{liminf}_{\xi<\lambda}\iota_{k,\xi})$ -- we fix $\mu:=\text{max}(\beta,\iota)+1$ as the first element of these sequences, which will never be changed.\footnote{See \cite{CarlBook}, p. 31-32 or  \cite{Koepkes Zoo II} for a detailed explanation of this trick.} Let $\delta:=f(\iota)$, $\delta^{\prime}:=f(\iota^{\prime})$ be the ordinals coded by $\iota$ in $c$ and by $\iota^{\prime}$ in $c^{\prime}$, respectively. 

We now perform two checks, one whether there is an order-preserving embedding of $\delta$ into $\delta^{\prime}$ and one for the reverse embedding. These checks will recursively call $P_{\text{compare}}$. The well-foundedness of $c$ and $c^{\prime}$ will ensure that the recursion terminates. 

To check whether $\delta$ embeds into $\delta^{\prime}$, we start with the sequences $(\mu,\iota)$ in $R$ and $(\mu,\iota^{\prime})$ in $R^{\prime}$. Now, for every $\xi<\beta$, we check whether $p(\xi,\iota)\in c$, i.e., whether $f(\xi)\in f(\iota)$. If that is the case, we replace the content of $R$ by $(\mu,\iota,\xi)$ and do the following: 
Searching through $\beta$, we test for each $\xi^{\prime}<\beta$ whether $p(\xi^{\prime},\iota^{\prime})\in c^{\prime}$ (i.e., whether $f^{\prime}(\xi^{\prime})\in f^{\prime}(\iota^{\prime})$. If that is the case, we replace the content of $R^{\prime}$ with $(\mu,\iota^{\prime},x^{\prime})$. If such a $\xi$, but no such $\xi^{\prime}$ is found, or vice versa, we output $0$; this means one, but not the other, of $\iota$, $\iota^{\prime}$ codes $0$, so that they do not code the same ordinal. Otherwise, we recursively call $P_{\text{compare}}$ to use $R$ and $R^{\prime}$ to check whether $f(\xi)=f^{\prime}(\xi^{\prime})$ (leaving the first three elements of the sequences stored in this register unchanged).  When this check terminates with output $0$, we know that $\xi$ and $\xi^{\prime}$ do not code the same ordinal, and so we proceed with the next candidate for $\xi^{\prime}$. When no candidate for $\xi^{\prime}$ are left -- i.e., when the search through $\beta$ has been completed without success --, we know that $\delta$ has an element that is not isomorphic to any element of $\delta^{\prime}$, so that we must have $\delta>\delta^{\prime}$ and output $0$. 
Otherwise, we continue with the next candidate for $\xi$. When all $\xi<\beta$ have been checked and the check has not terminated with a negative result, we know that $\delta\leq\delta^{\prime}$. We then proceed in exactly the same way to check whether $\delta^{\prime}\leq\delta$. 

When both of these checks terminate successfully, we halt with output $1$. 

%We split our check into two parts: Deciding whether the ordinal coded by $\iota$ in $c$ embeds into the

%Gegeben $\iota<\alpha$, $c\subseteq\beta$, wobei $\xi<\beta$ durch $\omega\xi$ codiert wird. Berechne einen Code f\"ur $\iota$: $c_{\iota}:=\{p(\iota_{1},\iota_{2}):\iota_{1}<\iota_{2}<\iota\}$. 
%Wir m\"ussen jetzt $\beta$ durchsuchen und f\"ur jedes $\omega\xi<\beta$ pr\"ufen, ob $c\upharpoonright\omega\xi:=\{p(\iota_{1},\iota_{2})\in c:p(\iota_{1},\omega\xi)\in c\wedge p(\iota_{2},\omega\xi)\in c\}$ pr\"ufen, ob $c\upharpoonright\omega\xi$ eine zur von $c_{\iota}$ codierten Struktur isomorphe Struktur codiert. 

%Dazu: Stack algorithm. Suche f\"ur jedes Element des einen Codes ein passendes aus dem anderen. Damit es bei Limites passt, erst einmal $\beta+1$ auf den Boden des Stacks, dann jeweils alle Elemente durchsuchen, und f\"ur die alle Elemente auf der n\"achsten Ebene pr\"ufen etc.\todo{DO IT!!!}
\end{proof}

\begin{lemma}{\label{lome wITRM}}
Let $\alpha$ be an exponentially closed ordinal such that, for some $\beta$, we have $\beta<\alpha<\sigma_{\beta}$. Then there is a lost melody for $\alpha$-wITRMs. 
\end{lemma}
\begin{proof}
%It is easy to see that 
By Lemma \ref{sigma sequence}, we have $\sigma_{\alpha}=\sigma_{\beta}$.
%beweis: nimm eine stabile stufe über \alpha (in der also erstmals eine formel mit parametern in \beta wahr wird. für die kriegen wir einen $\beta$-code $c$; aus $c$ und dem $\iota<\beta$, das in $c$ das $\alpha$ codiert, können wir $\alpha$ definieren (bzw. jede ordinalzahl <\alpha$) und damit dann alles, was man mithilfe solcher parameter definieren kann.

The statement $\psi$ ``There is an ordinal $\tau$ such that every $\alpha$-wITRM-program in every parameter $\rho<\alpha$ either halts, loops or overflows by time $\tau$'' is $\Sigma_{1}$ (since computations of length $<\tau$ are contained in $L_{\tau}$) and thus, since such $\tau$ exists, it is $<\sigma_{\alpha}$. 

By fine-structure, $L_{\sigma_{\beta}}$ contains a bijection $f:\beta\rightarrow\alpha$. 
Pick $\gamma\in(\alpha,\sigma_{\beta})$ such that $f\in L_{\gamma}$, and for some $\rho\in\beta$ and some $\in$-formula $\phi$, $\gamma$ is minimal with the property $L_{\gamma}\models\phi(\rho)$. Let $c$ be the $<_{L}$-minimal nice $\beta$-code for $L_{\gamma}$ and let $\xi<\beta$ be the ordinal that codes $f$ in the sense of $c$. 

By the techniques discussed in section \ref{itrms}, $c$ is $\alpha$-ITRM-recognizable in the parameter $\rho$. More precisely, the well-foundedness check specifically works as described in the proof of \cite{CarlBook}, Theorem 2.3.25(ii): One performs a depth-first search for an infinite descending chain through $\beta$, initially putting $\beta^{2}$ on the stack to ensure that the stack encoding works properly at limits. Store the length of the descending sequence currently built in a separate register, and output $1$ if that register contains $\omega$ at some point, and $0$ if the search terminates and that has not happened. The other checks necessary for recognizing $c$ in the way described in the proof of Lemma \ref{minimal name identification}, such as evaluating truth predicates in $\beta$-coded structures, can now be executed on a $\beta$-ITRM, which can be simulated on an $\alpha$-wITRM in the parameter $\alpha$ when $\alpha>\beta$.\footnote{Note that the extra complications in the well-foundedness are only necessary when $\beta$ is a regular cardinal in $L$ that has cofinality $\omega$ in $V$; otherwise, a well-foundedness check can be performed on a $\beta$-ITRM (possible relative to some recognizable oracle), which can be simulated on an $\alpha$-wITRM.} 
%}\todo{Was ist, wenn $\beta$ eine in $L$ reguläre Kardinalzahl mit Konfinalität $\omega$ ist?} %, and thus $\alpha$-wITRM-recognizable in the parameters $\rho$ and $\beta$.   

We claim that $c$ is $\alpha$-wITRM-recognizable (as a subset of $\alpha$) in the parameters $\beta$, $\rho$ and $\iota$. Thus, let a set $d\subseteq\alpha$ be given in the oracle.

First, let $\overline{d}:=d\cap \beta$, which is computable from $d$ in the parameter $\beta$. As just described, check whether $\overline{d}=c$. %perform a well-foundedness check on $d\cap\beta$
%First, use the parameter $\beta$ to simulate a $\beta$-ITRM-program that recognizes $c$ (as a subset of $\beta$) on an $\alpha$-wITRM and run this program on $d\cap\beta$. 
If the program returns $0$, we halt with output $0$. 

Otherwise, we know that the first $\beta$ many bits of $d$ are correct and it remains to check that there $d$ contains no elements that are greater than or equal to $\beta$. 

Using the parameter $\iota$ and the program $P_{\text{decode}}$ from Lemma \ref{identification lemma}, we can compute the bijection $f:\beta\rightarrow\alpha$ as follows: Given $\xi<\beta$, search through $c$ to find the (unique) element of the form $p(p(\xi,\zeta),\iota)$. Then run $P_{\text{decode}}^{c}(\zeta,\beta)$. 

Using $f$, we can now run through $\beta$ and check, for every $\xi\in\beta$, whether $f(\xi)\geq\beta$ and whether $f(\xi)\in c$. If the answer is positive for some $\xi\in\beta$, we halt with output $0$. Otherwise, we halt with output $1$.

%Idee: Nimm einen $\beta$-Code f\"ur ein hinreichend hohes $L$-level, in welchem zuerst eine gewisse Formel in Parametern $<\beta$ wahr wird und in dem $\alpha$ zugleich auf $\beta$ kollabiert wird. den kann man identifizieren, indem man ihn auf einer simulierten $\beta$-ITRM erkennen l\"a{\ss}t. Jetzt kann man HOFFENTLICH auch - ggf. unter gewissen zusatzbedingungen an den code? - zu einem $\iota<\beta$, welches ein element von $\alpha$ codiert, dieses element identifizieren. dann kann man jetzt eine singularisierung von $\alpha$ berechnen (und weiss, dass man sie hat), womit man nun pr\"ufen kann, dass das orakel \"uber $\beta$ hinaus keine elemente enth\"alt. (das "HOFFENTLICH" geht aber vermutlich NICHT, jedenfalls steht es nicht im buch. oder? m\"oglicherweise gibt korollar 2.3.31 das her, wenn man $\beta$ gr\"o{\ss}er macht?)
\end{proof}

We note the following amusing consequence which yields examples of a definability concept for which the sets of explicitly and implicitly definable objects are disjoint: 

\begin{corollary}
If $\alpha$ is $\Pi_{3}$-reflecting and $\alpha\in(\beta,\sigma_{\beta})$ for some ordinal $\beta$, then 
the set of $\alpha$-wITRM-computable subsets of $\alpha$ and the set of constructible $\alpha$-wITRM-recognizable subsets of $\alpha$ are (both non-empty, but) disjoint.
\end{corollary}
\begin{proof}
Immediate from Lemma \ref{lome wITRM} and Corollary \ref{comp not recog}.
\end{proof}

%...aber das schliesst ja nicht aus, dass es (konstruktible) lost melodies geben koennte. was fuer ein irres szenario!

%potenziell wichtiges kriterium: gibt es ein programm, das jedes element von $\alpha$ einmal in ein register schreibt, ohne \"uberzulaufen (in beliebiger reihenfolge, evtl. auch mit wiederholungen)? 
%ist das äquivalent zu ITRM-singularität?

%\begin{thm}{\label{witrm no cardinal}}
%If $\alpha$ is not a cardinal in $L$, then there are lost melodies for $\alpha$-wITRMs.
%\end{thm}
%\begin{proof}
%Let $\beta<\alpha$ be the $L$-cardinality of $\alpha$ and let $f$ be the $<_{L}$-minimal bijection %$f:\beta\rightarrow\alpha$, encoded as $\{p(\iota,f(\iota)):\iota<\alpha\}\subseteq\alpha$. 
%Le
%\end{proof}

It thus remains to consider the cases where $\alpha$ is either of the form $\sigma_{\beta}$ or a limit of ordinals of this form (note that the latter case generalizes the case of regular cardinals in $L$). 

\begin{lemma}{\label{witrm sigma case}}
If $\alpha$ is of the form $\sigma_{\beta}$ for some ordinal $\beta$, then there are no $\alpha$-wITRM-recognizable constructible subsets of $\alpha$ (and thus, in particular, no lost melodies for $\alpha$-wITRMs).
\end{lemma}
\begin{proof}
Let us first assume that $\alpha=\sigma_{\beta}$ for some $\beta\in\text{On}$. Suppose for a contradiction that $x\subseteq\alpha$ is $\alpha$-wITRM-recognizable by the program $P$ in the parameter $\rho<\alpha$. 
Thus, $L$ believes that there are a set $x\subseteq\alpha$ and a halting $\alpha$-wITRM-computation of $P^{x}(\rho)$ with output $1$. 
In particular, $L$ believes that there are a set $x\subseteq\text{On}$ and a halting ORM-computation of $P^{x}(\rho)$ with output $1$, which is a $\Sigma_1$-formula in the parameter $\rho$. 
By definition of $\sigma_{\beta}$, and the fact that $\sigma_{\beta}=\sigma_{\rho+1}$ (see Lemma \ref{sigma sequence} above), the same holds in $L_{\sigma}$. Since computations are absolute between transitive $\in$-structures, $L_{\sigma_{\beta}}$ contains a set $x$ of ordinals and a halting ORM-computation $P^{x}$ with output $1$. Let $\delta$ be the length of this computation. Then $\delta<\alpha$. Consequently, this computation cannot generate register contents $\geq\alpha$ and is thus actually a $\sigma$-wITRM-computation. During this computation, at most the first $\delta$ many bits of $x$ can be considered. It follows that both $P^{(x\cap\delta)}(\rho)$ and $P^{(x\cap\delta)\cup(\delta+1)}(\rho)$ halt in $\delta$ many steps with output $1$ without generating register contents $\geq\alpha$; thus, we have found two different oracles $y$ for which the $\alpha$-wITRM-computation $P^{y}(\rho)$ halts with output $1$, a contradiction to the assumption that $P$ recognizes $x$ in the oracle $y$. 
If $\alpha$ is a limit of ordinals of the form $\sigma_{\beta}$, pick $\beta\in\text{On}$ large enough such that $\rho\in\sigma_{\beta}$ and repeat the above argument. 
\end{proof}

This settles the question whether lost melodies exist for $\alpha$-wITRMs for all exponentially closed values of $\alpha$. 
We summarize the possible relations between COMP$^{\alpha}_{\text{wITRM}}$ and $L\cap$RECOG$^{\alpha}_{\text{wITRM}}$ that can occur for exponentially closed ordinals $\alpha$:

\begin{itemize}
%HIERHER!
\item COMP$^{\alpha}_{\text{wITRM}}=L\cap$RECOG$^{\alpha}_{\text{wITRM}}$ holds if and only if $\alpha=\omega$.
\item COMP$^{\alpha}_{\text{wITRM}}\subsetneq L\cap$RECOG$^{\alpha}_{\text{wITRM}}$ 
holds if and only if $\alpha>\omega$ is searchable and, for some $\beta$, 
$\beta<\alpha<\sigma_{\beta}$. 
\item COMP$^{\alpha}_{\text{wITRM}}\cap$RECOG$^{\alpha}_{\text{wITRM}}=\emptyset$ (with COMP$^{\alpha}_{\text{wITRM}}\neq\emptyset$, RECOG$^{\alpha}_{\text{wITRM}}\neq\emptyset$) holds if and only if $\alpha>\omega$ is not searchable and, for some $\beta$, $\beta<\alpha<\sigma_{\beta}$. 
\item RECOG$^{\alpha}_{\text{wITRM}}=\emptyset\subsetneq$COMP$^{\alpha}_{\text{wITRM}}$ holds if and only if $\alpha$ is stable. 
\end{itemize}

\begin{question}
    Under what conditions are there lost melodies for $\alpha$-wITRMs that are $\alpha$-ITRM-recognizable without parameters? Which relations can occur between the parameter-free versions of COMP$^{\alpha}_{\text{wITRM}}$ and RECOG$^{\alpha}_{\text{wITRM}}$?
\end{question}

By basically the same arguments, we obtain:

\begin{corollary}
If $\alpha$ is of the form $\sigma_{\beta}$ for some ordinal $\beta$, then there are no (weakly) $\alpha$-wITRM-semi-recognizable and no $\alpha$-wITRM-co-semi-recognizable constructible subsets of $\alpha$.
\end{corollary}
\begin{proof}
This works by the same argument as Lemma \ref{witrm sigma case}, noting that ``There is $x$ such that $P^{x}(\rho)$ halts'' is a $\Sigma_1$-formula and that ``There is $x$ such that $P^{x}(\rho)$ does not halt (but is defined)'' is equivalent to ``There is $x$ such that there is a strong loop in the computation of $P^{x}(\rho)$'', which is again $\Sigma_1$. 
\end{proof}

For weak $\alpha$-wITRM-co-semi-recognizability, however, things are different:

\begin{prop}
For all $\alpha$, each $\alpha$-wITRM-computable subset $x\subseteq\alpha$ is also $\alpha$-wITRM-co-semi-recognizable.
\end{prop}
\begin{proof}
Let $x\subseteq\alpha$ be $\alpha$-wITRM-computable, and pick a program $P$ and an ordinal $\xi<\alpha$ such that $P$ computes $x$ in the parameter $\xi$. Let $Q$ be the program that, for each $\iota<\alpha$, stored in some register $R$, computes $P(\iota,\xi)$ and compares the output to the $\iota$-th bit of the oracle. If they agree, $Q$ continues with $\iota+1$; otherwise, $Q$ halts. Clearly, $R$ will overflow at time $\alpha$ if and only if the oracle is equal to $x$, and otherwise, $Q$ will halt. 
\end{proof}

\begin{question}
    Suppose that $x\subseteq\alpha$ is $\alpha$-(w)ITRM-recognizable by the program $P$ and that $P^{x}$ halts in $\tau$ many steps. Does it follow that $x\in L_{\tau^{+\omega}}$ (where $\tau^{+\omega}$ denotes the next limit of admissible ordinals after $\tau$)? (This is true for $\alpha=\omega$ by Theorem 14 of \cite{ITRM Recog 2}.)
\end{question}

\end{document}